\documentclass[11pt]{article}
\usepackage{geometry}                
\usepackage{graphicx}
\usepackage{amssymb}
\usepackage{epstopdf}
\usepackage[T1]{fontenc}
\usepackage[latin1]{inputenc}
\usepackage[ngerman, english]{babel}
\usepackage{mathtools}
\usepackage{mathabx}
\usepackage{bbm}
\usepackage{mathrsfs}
\usepackage{amsmath,amscd}
\usepackage{tikz}
\usepackage{tikz-cd}
\usepackage{babel}
\usetikzlibrary{arrows, matrix}
\usetikzlibrary{cd}
\usepackage[arrow, matrix, curve]{xy}

\newcommand{\qed}{\hfill \ensuremath{\Box}}
\newenvironment{proof}{\vspace{1ex}\noindent{\it Proof.}\hspace{0.5em}}
	{\hfill\qed\vspace{1ex}}

\newtheorem{theorem}{Theorem}[section]
\newtheorem{lemma}[theorem]{Lemma}
\newtheorem{proposition}[theorem]{Proposition}
\newtheorem{corollary}[theorem]{Corollary}

\newtheorem{definition}[theorem]{Definition}

\DeclareMathOperator{\Gal}{\operatorname{Gal}}
\DeclareMathOperator{\Q}{\mathbf{Q}}
\DeclareMathOperator{\R}{\mathrm{R}}
\DeclareMathOperator{\Z}{\mathbf{Z}}

\DeclareMathOperator{\Ps}{\mathbf{P}}

\DeclareMathOperator{\N}{\mathbf{N}}

\DeclareMathOperator{\Spec}{\operatorname{Spec}}

\DeclareMathOperator{\Frac}{\operatorname{Frac}}

\DeclareMathOperator{\Id}{\mathrm{Id}}
\DeclareMathOperator{\Og}{\mathcal{O}}

\DeclareMathOperator{\Pic}{\mathrm{Pic}}

\DeclareMathOperator{\et}{\acute{\mathrm{e}}{\mathrm{t}}}

\DeclareMathOperator{\fppf}{\mathrm{fppf}}

\DeclareMathOperator{\Res}{\mathrm{Res}}

\DeclareMathOperator{\alg}{{^\mathrm{alg}}}
\DeclareMathOperator{\sep}{{^\mathrm{sep}}}

\title{On Jacobians of geometrically reduced curves and their Néron models}
\author{Otto Overkamp}
\date{}

\begin{document}
\maketitle 

{\abstract{We study the structure of Jacobians of geometrically reduced curves over arbitrary (i. e., not necessarily perfect) fields. We show that, while such a group scheme cannot in general be decomposed into an affine and an Abelian part as over perfect fields, several important structural results for these group schemes nevertheless have close analoga over imperfect fields. We apply our results to prove two conjectures due to Bosch-Lütkebohmert-Raynaud about the existence of Néron models and Néron lft-models over excellent Dedekind schemes in the special case of Jacobians of geometrically reduced curves. Finally, we prove some existence results for semi-factorial models and related objects for general geometrically integral curves in the local case.}}

\tableofcontents

\section{Introduction}
\subsection{Background}
Let $k$ be a field and let $C$ be a geometrically reduced connected proper algebraic curve over $k.$ The purpose of this article is to investigate the relationship between the structure of $C$ and the structure of its Jacobian $\Pic^0_{C/k}$ without imposing the condition that $k$ be perfect. In the case where $k$ does happen to be perfect, this has been worked out in great detail in the literature; see, for example, \cite{BLR}, Chapter 9.1. We shall begin by describing the situation over perfect ground fields; this will enable us to formulate more precise questions for the general case, which we shall subsequently answer. 
Suppose now that $k$ is a perfect field and that $C$ is a proper \it geometrically integral \rm algebraic curve over $k$. Let $\nu\colon \tilde{C}\to C$ be the normalisation morphism. By 
\cite{BLR}, p. 247, there is a unique factorisation $\tilde{C}\overset{\tilde{\nu}}{\to} C'\overset{\nu'}{\to} C$ of $\nu,$ such that $C'$ is the largest curve between $C$ and its normalisation which is universally homeomorphic to $C.$ Then we have the following
\begin{proposition} \rm (\cite{BLR}, Chapter 9.1, Propositions 9 and 10) \it The morphisms $$\Pic^0_{C/k}\overset{\nu'^\ast}{\to} \Pic^0_{C'/k} \overset{\tilde{\nu}^\ast}{\to}\Pic^0_{\tilde{C}/k}$$ are surjective in the étale topology and induce a filtration
$$0\subseteq \ker{\nu'^\ast}\subseteq \ker\nu^\ast  \subseteq \Pic^0_{C/k},$$ whose successive quotients are a smooth connected unipotent algebraic group, a torus, and an Abelian variety over $k,$ respectively. \label{IntroBLRProp}
\end{proposition}
The main observation here is that the filtration constructed above in terms of the morphism $\nu$ is \it intrinsic \rm to the algebraic group $\Pic^0_{C/k}.$ Indeed, since $k$ is perfect, Chevalley's theorem (\cite{CGP}, Theorem A.3.7), together with the well-known structure theory of smooth connected commutative affine algebraic groups over perfect fields (\cite{CGP}, Propositions A.1.4 and A.2.11, and \cite{SGA3}, Exposé XVII, Théorème 6.1.1 A) ii)), tells us that there is a unique exact sequence
$$0\to U\times_k T\to \Pic^0_{C/k}\to A\to 0$$ (which depends only upon $\Pic^0_{C/k}$ and not on $C$), where $U,$ $T$, and $A$ are a smooth connected unipotent algebraic group, a torus, and an Abelian variety over $k,$ respectively. In the notation of Proposition \ref{IntroBLRProp}, we have $\ker{\nu'^\ast}=U,$ $\ker\nu^\ast /\ker{\nu'^\ast}=T,$ and $\Pic^0_{C/k}/ \ker \nu^\ast =A.$ It is well-known that the factorisation $\nu=\nu'\circ \tilde{\nu},$ Chevalley's theorem, as well as most statements of the structure theory of smooth connected commutative affine algebraic groups over perfect fields, all fail if we drop the condition that $k$ be perfect. We shall see however, that there is a way of describing the structure of the Jacobian of a geometrically reduced curve over an imperfect field which closely resembles the situation over a perfect field (see Section \ref{JacSecY}). More precisely, we shall see in Theorem \ref{Jacstructthm} that, over an arbitrary field $\kappa,$ if $C^{\mathrm{sn}}$ denotes the \it seminormalisation  \rm of $C$ (see Proposition \ref{seminormalisationexprop}), we have a factorisation 
$$\widetilde{C}\overset{\widetilde{\varsigma}}{\to} C^{\mathrm{sn}}\overset{\varsigma}{\to} C$$ of the normalisation morphism $\nu\colon \widetilde{C}\to C,$ which induces a filtration
$$0\subseteq \ker \varsigma^\ast \subseteq \ker \nu^\ast \subseteq \Pic^0_{C/\kappa},$$
such that $\ker \varsigma^\ast$ equals the maximal smooth connected split unipotent group of $\Pic^0_{C/\kappa}$ (cf. \cite{CGP}, p. 63), and such that $\ker \nu^\ast$ equals the maximal smooth unirational subgroup of $\Pic^0_{C/\kappa}$ (cf. \cite{BLR}, p. 310). Observe that, over a perfect field, a smooth connected commutative algebraic group is unirational if and only if it is affine, so this result recovers the filtration in the perfect case quoted above. We shall see that $C^{\mathrm{sn}}\to C$ is still a universal homeomorphism (Proposition \ref{seminormalisationexprop}), but it is no longer the largest curve between $C$ and $\widetilde{C}$ which is universally homeomorphic to $C.$ \\
Having studied the structure of Jacobians of geometrically reduced proper curves in the general case, we apply our results in order to prove two conjectures due to Bosch-Lütkebohmert-Raynaud (\cite{BLR}, Chapter 10.3, Conjecture I and Conjecture II) for Jacobians of such curves. The crucial observation we shall use to investigate the structure of Jacobians is the Factorisation Theorem (see Theorem \ref{FactorisationTheorem}). This result will imply in particular that \it all \rm singularities of curves can be obtained by repeatedly applying a push-out construction, beginning with a regular curve. In \cite{Ov}, the author introduced a method to use the push-out construction in order to construct proper flat models of singular curves which are well-suited to studying Picard functors and Néron models of Jacobians. The Factorisation Theorem makes it possible to apply this construction to the study of the Jacobian of \it any \rm singular curve over a field. In order to make the construction from \cite{Ov} fit for our purposes, we must generalise it in several directions, which will be accomplished in Paragraphs \ref{pushoputpara}, \ref{factoristionpara}, and Subsection \ref{twopicsubsec}. By constructing suitable proper flat models of singular curves over Dedekind schemes using the push-out construction, we prove
\begin{theorem} \rm (Conjecture II; Theorem \ref{ConjIIthm}) \it
Let $S$ be an excellent Dedekind scheme with field of rational functions $K.$ Let $C$ be a proper geometrically reduced curve over $K.$ Assume that $\Pic^0_{C/K}$ contains no closed subgroups which are unirational. Then $\Pic^0_{C/K}$ admits a Néron model over $K,$ \label{ConjIIthmInt}
\end{theorem}
as well as  
\begin{theorem} \rm (Conjecture I; Theorem \ref{ConjIthm}) \it 
Let $S$ be an excellent Dedekind scheme with field of rational functions $K.$ Let $C$ be a proper geometrically reduced curve over $K,$ and suppose that $\Pic^0_{C/K}$ contains no closed subgroup isomorphic to $\mathbf{G}_{\mathrm{a}, K}.$ Then $\Pic^0_{C/K}$ admits a Néron lft-model over $S.$  \label{ConjIthmInt}
\end{theorem}
These two Conjectures were previously known for smooth connected algebraic groups of dimension 1 (\cite{LiuTong}, Corollary 7.8, Remark 7.9). Moreover, Conjecture II is known for smooth connected algebraic groups which admit a regular compactification (\cite{BLR}, Chapter 10.3, Theorem 5). While compactifications of Jacobians have been studied by many authors, there do not seem to be any results on \it regular \rm compactifications of Jacobians in positive characteristic which are general enough for our purposes.\\
Finally, we use the techniques developed in this article in order to construct \it semi-factorial models \rm (cf. \cite{Pép}) and \it Néron-Picard models \rm of geometrically integral (possibly singular) curves. This will allow us to write the Néron model of the Jacobian of a geometrically integral seminormal curve in terms of the Picard functor of a particular proper flat model of the curve, which generalises earlier well-known results for regular curves. \\
\\
$\mathbf{Acknowledgement}.$ The author would like to thank the Mathematical Institute of the University of Oxford, where this paper was written, for its hospitality. He would like to express his gratitude to Professor D. Rössler for helpful conversations and to Professor Q. Liu for bringing the paper \cite{LiuTong} to his attention. Moreover, the author was supported by the German Research Foundation (Deutsche Forschungsgemeinschaft; Geschäftszeichen OV 163/1-1, Projektnummer 442615504), for whose contribution he is most grateful. Finally, the author would like to thank the referees for their very careful reading of this paper and for making a large number of suggestions which led to considerable improvements. Several proofs have been simplified (in particular, those of Theorems \ref{FactorisationTheorem} and \ref{Jacstructthm}), and the exposition improved significantly as a consequence of those suggestions. 

\subsection{Notation and conventions}
We fix some notation and state precisely a few definitions which are not applied uniformly in the literature. 
\begin{itemize}
\item When we speak of \it algebraic spaces, \rm we use Definition 4 from \cite{BLR}, Chapter 8.3, which goes back to Knutson. In particular, an algebraic space $\mathscr{X}\to S$ over a scheme $S$ is, by definition, locally of finite presentation and locally separated over $S.$ Note that the definition of an algebraic space used in \cite{Stacks} is more general. 
\item Let $S$ be a scheme. We shall say that an effective Cartier divisor $D$ on $S$ has \it strict normal crossings \rm if it has strict normal crossings in the sense of \cite{Stacks}, Tag 0BI9. In particular, a divisor with strict normal crossings is reduced. The reader should bear in mind that the terminology used in \cite{Liu}, Chapter 9.1, Definition 1.6 is different: a divisor is said to have \it normal crossings \rm in \it op. cit. \rm if and only if it is \it supported on a strict normal crossings divisor \rm in the sense of \cite{Stacks}, Tag 0CBN. 
\item A \it Dedekind scheme \rm is a regular separated quasi-compact scheme of pure dimension 1. Unless indicated otherwise, we shall assume that Dedekind schemes are connected. 
\item Let $S$ be a Dedekind scheme with field of rational functions $K$ and let $G$ be a smooth commutative group scheme over $K.$ A \it Néron model \rm is a smooth separated model $\mathscr{G}\to S$ of $G$ which satisfies the Néron mapping property (\cite{BLR}, Chapter 1.2, Definition 1), and which is \it of finite type over $S.$ \rm A \it Néron lft-model \rm of $G$ is a smooth separated model $\mathscr{G}\to S$ of $G$ which satisfies the Néron mapping property. This means that we follow the terminology of \cite{BLR} (\it op. cit., \rm Chapter 1.2, Definition 1 and Chapter 10.1, Definition 1) . Some authors use the terms \it Néron model \rm and \it Néron lft-model \rm interchangeably.
\item For a morphism of schemes $X\to Y,$ we denote by $X^{\mathrm{sm}}$ the set of points of $X$ at which the morphism is smooth. This is an open subset of $X$ almost by definition (\cite{Stacks}, Tag 01V5). 
\item For a field $k,$ we denote by $k\sep$ a choice of separable closure of $k.$ 
\end{itemize}

\section{The structure of Jacobians over general fields}
\label{JacSecY}
\subsection{Classification of prime algebras}

Let $\kappa$ be an arbitrary field. In this subsection, we shall classify what we call \it prime algebras \rm over $\kappa,$ generalising a result from \cite{Stacks}. 

\begin{definition}
Let $A$ be an algebra over $\kappa.$ We say that $A$ is a \rm prime algebra \it over $\kappa$ if the map $\kappa \to A$ is not surjective (in particular, $A\neq0$), and the only $\kappa$-subalgebras of $A$ are $\kappa$ and $A$.
\end{definition}
In \cite{Stacks}, Tag 0C1I, it is shown that, if $\kappa$ is algebraically closed (hence perfect), then any prime algebra over $\kappa$ is isomorphic to $\kappa\times\kappa$ or to $\kappa[\epsilon]/\langle \epsilon^2\rangle.$ We shall now generalise this result: 

\begin{proposition}  \label{Primeclassprop}
Let $\kappa$ be an arbitrary field and let $A$ be a prime algebra over $\kappa.$ Then $A$ is isomorphic, as a $\kappa$-algebra, to precisely one of the following: \\
(i) $\kappa[\epsilon]/\langle \epsilon^2 \rangle,$\\
(ii) $\kappa\times \kappa,$\\
(iii) $\kappa(a^{1/p})$ where $p=\mathrm{char}\, \kappa>0$ and $a\in k\backslash k^p,$\\ 
(iv) a finite non-trivial separable extension of $\kappa$ with no proper subextensions.\\
In particular, the map $\Spec A \to \Spec \kappa$ is an homeomorphism with trivial residue field extension if and only if $A$ is of type (i).
\end{proposition}
\begin{proof}
A prime algebra $A$ over $\kappa$ must be finite  over $\kappa.$ Indeed, any element in $A$ not contained in $\kappa$ gives rise to a surjective morphism $\varpi\colon \kappa[t]\to A.$ If $\ker \varpi=0,$ then $\varpi$ is an isomorphism. However, $\varpi(\kappa[t^2])$ is then a proper subalgebra of $A$ not equal to $k.$ Hence $\ker\varpi$ is generated by a non-zero polynomial of degree $m,$ say. Then $\dim_{\kappa}A=m.$ Assume first that $A$ is non-reduced. Let $y$ be a non-zero nilpotent element of $A.$ Then $A=\kappa[y].$ If $y^2\neq0,$ then $y^2$ generates a proper subalgebra of $A.$ Hence $A=\kappa[y]\cong\kappa[\epsilon]/\langle \epsilon^2 \rangle.$
Now suppose that $A$ is reduced. Because $A$ is finite over $\kappa$ and hence \it a fortiori \rm an Artinian ring, we can write $A$ as
$$A\cong A_1\times ...  \times A_r$$ for some $r\geq 1,$ where the $A_j$ are finite field extensions of $\kappa.$ If $r\geq 2,$ we let $\Delta$ be the image of the map $\kappa\to A_1\times A_2.$ If $r>2,$ then $\Delta\times A_3\times...\times A_r$ is a proper subalgebra of $A.$ Hence we must have $r\leq 2.$ If $r=2,$ we claim that we must have $A_1=\kappa=A_2.$ Indeed, otherwise we may assume without loss of generality that $\kappa\subseteq A_1$ is a proper inclusion, in which case $\kappa\times A_2$ would be a proper subalgebra of $A.$ Hence $A=\kappa\times \kappa.$ Finally, suppose that $r=1.$ Then $A=A_1$ is a finite field extension of $\kappa.$ Observe that $A$ must be either separable or purely inseparable over $\kappa,$ for otherwise the separable closure of $\kappa$ in $A$ would be a proper subalgebra. In the latter case, choose $\alpha\in A\backslash \kappa.$ Then $A=\kappa(\alpha).$ If $p:=\mathrm{char}\, \kappa$ and $\alpha^p\not\in \kappa,$ then $\alpha^p$ would generate a proper subextension of $\kappa\subseteq A.$ Hence $a:=\alpha^p\in \kappa$ and $A=\kappa(a^{1/p}).$ In the former case, $A$ is a finite separable extension of $\kappa$ which is non-trivial and admits no proper subextensions by assumption. 
\end{proof}\\
$\mathbf{Remark.}$ If $\kappa$ is separably closed, then the Proposition above gives a complete classification of prime algebras over $\kappa$ in the sense that we can give a precise description, in terms of generators and relations, of each prime algebra. Unfortunately, the problem of giving such a description for a general prime algebra over a field $\kappa$ which is not separably closed seems to be an intractable problem, even if $\mathrm{char}\, \kappa=0.$ For example, it is not the case that a finite separable extension $\kappa\subseteq L$ which is non-trivial and admits no proper subextensions must have prime degree. Indeed, let $\kappa=\Q$ and let $L$ be a finite Galois extension with $\Gal(L/\Q)\cong A_4.$ It is well-known that such extensions exist. Moreover, it is an elementary exercise to show that $A_4$ contains no subgroup of order 6. In particular, any subgroup of $A_4$ generated by a 3-cycle is maximal. Let $A$ be the subextension of $\Q\subseteq L$ which corresponds to such a subgroup under the Galois correspondence. By Galois theory, $A$ is a prime algebra over $\Q$ of degree 4. If one replaces $A_4$ by $A_5,$ one can even construct examples where the degree has more than one prime factor. 

\subsection{Some results on unirational algebraic groups}
Let $G$ be a smooth connected affine commutative algebraic group over an arbitrary field $\kappa.$ It is well-known (\cite{SGA3}, Exposé XIV, Corollaire 6.10) that, if $\kappa$ is perfect, then $G$ is unirational by (i. e., there is a dominant morphism $U \to G$ with $U$ an open subscheme of $\mathbf{A}^n_{\kappa}$ for some $n\in \N$). It does not seem to be known whether a commutative extension of two commutative unirational algebraic groups is again unirational (but see \cite{Achet}, Section 2.4, for some results in this direction). The following result is known (see, e. g., \cite{Borel}, Chapter V, Corollary 15.8). We give a proof in the commutative case for the reader's convenience: 
\begin{lemma}
Let $\kappa$ be an arbitrary field and let $G,$ $G',$ and $G''$ be smooth, connected, commutative algebraic groups over $\kappa.$ Assume that $G''$ is unirational, and that $G'$ is split unipotent (i. e., that $G'$ admits a composition series whose successive quotients are isomorphic to $\mathbf{G}_{\mathrm{a}, \kappa}$). Moreover, assume that there exists an exact sequence
$$0\to G'\to G\to G''\to 0$$ in the fppf-topology over $\kappa.$ Then $G$ is unirational.  \label{unirationallemma}
\end{lemma}

\begin{proof}
We prove the statement by induction on $\dim G'.$ If $\dim G'=1,$ then $G'\cong \mathbf{G}_{\mathrm{a}, \kappa}.$ It is well-known that the canonical morphism $H^1_{\mathrm{Zar}} (G'',\mathbf{G}_{\mathrm{a}, \kappa})\to H^1_{\fppf}(G'', \mathbf{G}_{\mathrm{a}, \kappa})$ is an isomorphism (\cite{Stacks}, Tag 03P2, which applies because $\mathbf{G}_{\mathrm{a}, \kappa} = \Og_{G''} ^  a$ in the terminology of \it loc. cit.\rm). Moreover, $G''$ is affine (since it is unirational\footnote{That $G''$ is affine can be checked over an algebraic closure $\kappa\alg$ of $\kappa.$ Then we have an exact sequence $0 \to L\to G''_{\kappa\alg} \to E \to 0,$ where $L$ is a smooth connected affine algebraic group over $\kappa\alg$  and $E$ an Abelian variety over that same field (\cite{CGP}, Theorem A.3.7). Since $G''_{\kappa\alg}$ is unirational, so is $E.$ However, it is a standard fact that there are no non-constant rational functions from $\Ps^1_{\kappa\alg}$ to $E$ (see the proof of Proposition \ref{unitrivialprop} below), so $E=0$ (\cite{BLR}, Chapter 10.3, Theorem 1) and $G''_{\kappa\alg}$ is affine.}), so $H^1_{\mathrm{Zar}}(G'', \mathbf{G}_{\mathrm{a}, \kappa})=H^1(G'', \Og_{G''})=0.$ In particular, the $\mathbf{G}_{\mathrm{a}, \kappa}$-fibration $G\to G''$ is trivial in the Zariski topology. This implies that $G$ is isomorphic (as a scheme) to $\mathbf{G}_{\mathrm{a}, \kappa}\times_\kappa G'',$ which clearly means that $G$ is unirational.\\
Now consider the general case. Because $G'$ is a repeated extension of $\mathbf{G}_{\mathrm{a}, \kappa},$ we can find a closed immersion $\mathbf{G}_{\mathrm{a}, \kappa}\to G'.$ Consider the exact sequences 
$$0\to G'/\mathbf{G}_{\mathrm{a}, \kappa}\to G/\mathbf{G}_{\mathrm{a}, \kappa}\to G''\to 0$$
and
$$0\to \mathbf{G}_{\mathrm{a}, \kappa}\to G\to G/\mathbf{G}_{\mathrm{a}, \kappa}\to 0.$$
By the induction hypothesis, we know that $G/\mathbf{G}_{\mathrm{a}, \kappa}$ is unirational from the first exact sequence. The same argument as above now shows that $G$ is unirational, using the second exact sequence. 
\end{proof}
\subsection{Néron models over Dedekind schemes}
Let $S$ be a Dedekind scheme with field of rational functions $K.$ Let $g\colon \mathscr{G}\to S$ be a smooth separated group scheme over $S.$ If $R$ is a discrete valuation ring and $S=\Spec R,$ then there is a convenient criterion which allows us to check whether $\mathscr{G}$ is the Néron lft-model of its generic fibre: by \cite{BLR}, Chapter 10.1, Proposition 2, this is the case if and only if for all local extensions $R\subseteq R'$ with $R'$ of discrete valuation rings essentially smooth over $R$ and $K':=\Frac R',$ the canonical map $\mathscr{G}(R')\to \mathscr{G}(K')$ is surjective. We shall need a slightly stronger criterion (see Proposition \ref{strongcritprop} below). Moreover, we shall consider the case where $S$ is allowed to have infinitely many closed points. \\
Recall that a local extension $R\subset R'$ of discrete valuation rings is said to be \it of ramification index one \rm if the maximal ideal of $R$ generates that of $R'$ and, moreover, the induced extension of residue fields is separable (i. e., geometrically reduced; we are not assuming that it is algebraic). 

\begin{proposition} Suppose that $S=\Spec R$ for some discrete valuation ring $R$ and that $g\colon \mathscr{G}\to S$ is a smooth separated group scheme. Then the following are equivalent: \\
(i) $\mathscr{G}$ is the Néron lft-model of its generic fibre,\\
(ii) for all local extensions $R\subseteq R'$ of discrete valuation rings of ramification index one with $K':=\Frac R'$ and $R'$ \rm strictly Henselian, \it the canonical map $\mathscr{G}(R')\to \mathscr{G}(K')$ is surjective, and \\
(iii) the canonical map $\mathscr{G}(R')\to \mathscr{G}(K')$ is surjective for all local extensions $R\subseteq R'$ such that $K'=\Frac R'$ and such that there is a filtration $R\subseteq R''\subseteq R'$ with $R''$ a discrete valuation ring essentially smooth over $R$ and such that $R''\subseteq R'$ is a strict Henselisation. \label{strongcritprop}
\end{proposition}

\begin{proof}
(iii) $\Rightarrow$ (i): Let $R\subseteq R''$ be a local extension of discrete valuation rings, and suppose that $R''$ is essentially smooth over $R.$ Let $K'':=\Frac R''.$ Let $R':=R''^{\mathrm{sh}}$ be the strict Henselisation of $R''$ with respect to some choice of separable closure of the residue field of $R'',$ and let $K''^{\mathrm{sh}}$ be its field of fractions. Let $x\colon \Spec K''\to \mathscr{G}$ be a morphism over $S.$ The induced morphism $x^{\mathrm{sh}}\colon \Spec K''^{\mathrm{sh}}\to \mathscr{G}$ comes from a morphism $y^{\mathrm{sh}}\colon \Spec R''^{\mathrm{sh}}\to \mathscr{G}$ by assumption. Let $U$ be an open affine neighbourhood in $\mathscr{G}$ of the image of the special point of $\Spec R''^{\mathrm{sh}}.$ Then $y^{\mathrm{sh}}$ factors through $U.$ Now consider the induced morphism $\Gamma(U, \Og_U)\to R''^{\mathrm{sh}}.$ Because $x^{\mathrm{sh}}$ comes form a $K''$-point of $\mathscr{G},$ this morphism factors through $R''$ (indeed, $R''^{\mathrm{sh}}\cap K''=R''$), which implies that $y^{\mathrm{sh}}$ comes from a morphism $y\colon \Spec R''\to \mathscr{G}$ extending $x.$ We conclude this implication using \cite{BLR}, Chapter 10.1, Proposition 2. \\
The implication (i) $\Rightarrow$ (ii) follows from \cite{BLR}, Chapter 10.1, Proposition 3. \\
Finally, the implication (ii) $\Rightarrow$ (iii) is trivial since both $R\subseteq R''$ and $R''\subseteq R'$ are of ramification index 1; hence so is $R\subseteq R''.$
\end{proof}
\begin{lemma} \rm (cf. \cite{LiuTong}, Corollary 2.5) \it
Let $S$ be a Dedekind scheme and let $\mathscr{G}\to S$ be a smooth separated group scheme over $S.$ Suppose that, for all closed points $\mathfrak{p}$ of $S$, the group scheme $\mathscr{G}_{\mathfrak{p}}:=\mathscr{G}\times_SS_{\mathfrak{p}}$ is the Néron lft-model of its generic fibre, where $S_{\mathfrak{p}}$ is the localisation of $S$ at $\mathfrak{p}.$ Then $\mathscr{G}$ is the Néron lft-model of its generic fibre over $S.$ \label{localgloballem}
\end{lemma}
\begin{proof}
Let $K$ be the field of rational functions of $S.$ It suffices to show that, for all smooth morphisms $T\to S$ \it of finite presentation \rm  and every morphism $T_K\to \mathscr{G}_K,$ there is a unique morphism $T\to \mathscr{G}$ extending $T_K\to \mathscr{G}_K.$ Suppose we have chosen such a scheme and a morphism over $K.$ By passing to the limit (\cite{Stacks}, Tag 01ZC), there is a finite set of closed points $\{\mathfrak{p}_1, ..., \mathfrak{p}_r\}$ of $S$ such that $T_K\to \mathscr{G}_K$ extends to a morphism over $S\backslash \{\mathfrak{p}_1, ..., \mathfrak{p}_r\}.$ By assumption, $T_K\to \mathscr{G}_K$ also extends to morphisms $T\times_SS_{\mathfrak{p}_j}\to \mathscr{G}\times_SS_{\mathfrak{p}_j}$ for all $j=1,..., r.$ Because schemes are sheaves in the fpqc-topology, $T_K\to \mathscr{G}_K$ does indeed extend to a morphism $T\to \mathscr{G}$ as required. Uniqueness follows because $\mathscr{G}$ is separated over $S.$
\end{proof}\\
We can use Proposition \ref{strongcritprop} and Lemma \ref{localgloballem} above to deduce the following (partial) generalisation of \cite{BLR}, Chapter 7.5, Proposition 1 (b): 
\begin{corollary}
Let $S$ be a Dedekind scheme, and let $0\to \mathscr{G}'\to \mathscr{G}\to \mathscr{G}''\to 0$ be an exact sequence of smooth separated group schemes over $S.$ Assume, moreover, that $\mathscr{G}'$ and $\mathscr{G}''$ are the Néron lft-models of their respective generic fibres. Then so is $\mathscr{G}.$ \label{exactsequenceNéroncor}
\end{corollary}
\begin{proof}
By Lemma \ref{localgloballem}, we may assume, without loss of generality, that $S=\Spec R,$ where $R$ is a discrete valuation ring. 
In this case, Proposition \ref{strongcritprop} tells us that it suffices to show that for all local extensions $R\subseteq R'$ of discrete valuation rings of ramification index one with $R'$ strictly Henselian, the induced map $\mathscr{G}(R')\to \mathscr{G}(K')$ is surjective, where $K':=\Frac R'.$ Note that the sequence $0\to \mathscr{G}'(R')\to \mathscr{G}(R')\to \mathscr{G}''(R')\to 0$ is exact because $R'$ is strictly Henselian and $\mathscr{G'}$ is smooth over $S,$ as is the sequence $0\to \mathscr{G}'(K')\to \mathscr{G}(K')\to \mathscr{G}''(K').$ The same diagram chasing argument as in \cite{BLR}, Chapter 7.5, proof of Proposition 1 (b) shows that the map $\mathscr{G}(R')\to \mathscr{G}(K')$ is surjective, as desired.
\end{proof}\\
Many of the following Lemmata are certainly well-known to the experts; we give proofs here for the reader's convenience:
\begin{lemma}
Let $S$ be a locally Noetherian scheme of dimension $\leq 1.$ Let $S'\to S$ be a finite and locally free morphism. Let $\mathscr{G}\to S'$ be a separated group scheme locally of finite presentation over $S'.$ Then $\Res_{S'/S}\mathscr{G}$ is representable by a separated group scheme locally of finite presentation over $S.$ If $\mathscr{G}$ is smooth over $S',$ then $\Res_{S'/S}\mathscr{G}$ is smooth over $S.$ 
\end{lemma}

\begin{proof}
By \cite{BLR}, Chapter 7.6, Theorem 4, the functor $\Res_{S'/S}\mathscr{G}$ is representable as soon as any finite set of points of $\mathscr{G}$ contained in a fibre of $\mathscr{G}\to S'\to S$ is contained in an open affine subset of $\mathscr{G}.$ By \cite{An}, Théorème 4.A, the morphism $\mathscr{G}\to S'$ is \it de type (FA) \rm in the terminology of \it loc. cit., \rm i. e., every finite set of points of $\mathscr{G}$ which maps to an open affine subset of $S'$ is contained in an open affine subset of $\mathscr{G}.$ Let $P$ be a finite set of points of $\mathscr{G}$ all of whose elements are mapped to the point $s\in S.$ Let $U$ be an open affine neighbourhood of $s.$ Then the pre-image $V$ of $U$ in $S'$ is affine (since $S'$ is finite over $S$), and clearly $P$ is mapped into $V.$ Therefore $P$ is contained in an open affine subset of $\mathscr{G}.$ Hence $\Res_{S'/S}\mathscr{G}$ is indeed representable. The remaining claims follow from \cite{BLR}, Chapter 7.6, Proposition 5 (b), (d), and (h). 
\end{proof}\\
It is an immediate consequence of the Néron mapping property that Néron lft-models commute with Weil restriction, i. e., if $S'\to S$ is a finite locally free extension of Dedekind schemes and $\mathscr{G}\to S'$ is a group scheme which is the Néron lft-model of its generic fibre, then the functor $\Res_{S'/S}\mathscr{G}$ satisfies the Néron mapping property as well. The preceding Lemma shows that the Weil restriction $\Res_{S'/S} \mathscr{G}$ always exists and satisfies the scheme-theoretic properties required of a Néron lft-model. We shall use this result freely throughout this paper.\\
\\
In some cases, it is possible to construct Néron lft-models by hand; the most prominent example is the Néron lft-model of ${\mathbf{G}_{\mathrm{m}, K}}$ over a Dedekind scheme $S$ with field of rational functions $K$ (see \cite{BLR}, Chapter 10.1, Example 5), which we shall denote by $\mathscr{G}_{\mathrm{m}, S}.$ 

\begin{lemma} \label{Gmclimprop}
Let $S$ be an excellent Dedekind scheme with field of rational functions $K$ and let $G$ be a smooth connected algebraic group scheme over $K.$ Suppose that there is a closed immersion ${\mathbf{G}_{\mathrm{m}, K}}\to G$ of $K$-group schemes. Assume, moreover, that $G$ admits a Néron lft-model $\mathscr{G}$ over $S.$ Then the induced morphism $\mathscr{G}_{\mathrm{m}, S}\to \mathscr{G}$ is a closed immersion. Moreover, the fppf-quotient $\mathscr{G}/\mathscr{G}_{\mathrm{m}, S}$ is representable and isomorphic to the Néron lft-model of $G/{\mathbf{G}_{\mathrm{m}, K}}.$
\end{lemma}
\begin{proof}
Note that $G/\mathbf{G}_{\mathrm{m}, K}$ is representable by a $K$-group scheme of finite presentation by \cite{An}, Théorème 4.C. Moreover, this quotient is smooth because $G$ is smooth by assumption. Let $\mathfrak{p}$ be a closed point of $S$ and let $S_{\mathfrak{p}}$ be the localisation of $S$ at $\mathfrak{p}.$ Then $S_{\mathfrak{p}}$ is the spectrum of an excellent discrete valuation ring with field of fractions $K.$ Moreover, $\mathscr{G}_{\mathfrak{p}}:=\mathscr{G}\times_SS_{\mathfrak{p}}$ is the Néron lft-model of $G$ over $S_{\mathfrak{p}};$ the same is true for $\mathscr{G}_{\mathrm{m}, \mathfrak{p}}:=\mathscr{G}_{\mathrm{m},S}\times_SS_{\mathfrak{p}}$ and ${\mathbf{G}_{\mathrm{m}, K}}.$ First we claim that $G/{\mathbf{G}_{\mathrm{m}, K}}$ admits a Néron lft-model over $S_{\mathfrak{p}}.$ By \cite{BLR}, Chapter 10.2, Theorem 2 (b'), all we have to show is that $G/{\mathbf{G}_{\mathrm{m}, K}}$ does not have a closed subgroup isomorphic to $\mathbf{G}_{\mathrm{a}, K}$ (this is where we use that $S_{\mathfrak{p}}$ is excellent). If this were false, then, denoting by $G'$ the pre-image of $\mathbf{G}_{\mathrm{a},K}$ in $G,$ we would obtain an exact sequence $0\to {\mathbf{G}_{\mathrm{m}, K}}\to G'\to \mathbf{G}_{\mathrm{a},K}\to 0$ over $K.$  By \cite{SGA3}, Exposé XVII, Théorème 6.1.1 A) ii), we could now construct a closed immersion $\mathbf{G}_{\mathrm{a}, K}\to G'\to G,$ contradicting \cite{BLR}, Chapter 10.1, Proposition 8.\\
Let $\mathscr{G}'$ denote the Néron lft-model of $G/{\mathbf{G}_{\mathrm{m}, K}}$ over $S_{\mathfrak{p}}.$ Then the argument given in the proof of \cite{Chai}, Corollary 4.7 shows that the sequence
$$0\to \mathscr{G}_{\mathrm{m}, \mathfrak{p}}\to \mathscr{G}_{\mathfrak{p}}\to \mathscr{G}'\to 0$$ is exact. In particular, the map $\mathscr{G}_{\mathrm{m}, S}\to \mathscr{G}$ is a closed immersion after localising at any closed point $\mathfrak{p}$ of $S.$ Since this morphism is clearly locally of finite presentation, we find that it is unramified and universally injective (since it is set-theoretically injective on all fibres and induces isomorphisms on residue fields; see \cite{Stacks}, Tags 01S3, 01S4, and 02G8). By \cite{Stacks}, Tag 04XV (5), all that remains to be shown is that the morphism is universally closed. We use the valuative criterion for universal closedness (\cite{Stacks}, Tag 01KF). The map $\mathscr{G}_{\mathrm{m}, S} \to \mathscr{G}$ is easily seen to be quasi-compact since the group of components of $\mathscr{G}_{\mathrm{m}}$ is torsion-free at each point of $S.$ Let $\mathcal{R}$ be a valuation ring with field of fractions $\mathcal{K}$ and let $\varphi\colon \Spec \mathcal{R} \to \mathscr{G}$ be a morphism of schemes whose restriction to $\Spec \mathcal{K}$ factors through $\mathscr{G}_{\mathrm{m}, S}.$ Let $\mathfrak{m}$ be the maximal ideal of $\mathcal{R}.$ Let $\mathfrak{p}$ be the image of $\mathfrak{m}$ in $S.$ Then $\varphi$ factors through $\mathscr{G}_{\mathfrak{p}}.$ Moreover, the restriction to $\Spec \mathcal{K}$ of the induced map $\Spec \mathcal{R}\to \mathscr{G}_{\mathfrak{p}}$ factors through $\mathscr{G}_{\mathrm{m}, \mathfrak{p}}$ by assumption. Because we already know that the map $\mathscr{G}_{\mathrm{m}, \mathfrak{p}}\to \mathscr{G}_{\mathfrak{p}}$ is a closed immersion, we deduce that $\phi$ factors through $\mathscr{G}_{\mathrm{m}, S}.$ Hence we conclude that the morphism $\mathscr{G}_{\mathrm{m}, S}\to \mathscr{G}$ is a closed immersion. The claim that $\mathscr{G}/\mathscr{G}_{\mathrm{m}, S}$ is representable now follows from \cite{An}, Théorème 4.C, and it is easy to check that the quotient is smooth and separated over $S.$ Finally, this quotient is the Néron lft-model of $G/{\mathbf{G}_{\mathrm{m}, K}}$ over $S$ by Lemma \ref{localgloballem}.
\end{proof}

\begin{lemma}
Let $S$ be an excellent Dedekind scheme with function field $K.$ Let $A\subseteq B$ be two finite non-zero reduced $K$-algebras. Then the $K$-group scheme $$(\Res_{B/K}{\mathbf{G}_{\mathrm{m}, B}})/\Res_{A/K}{\mathbf{G}_{\mathrm{m}, A}}$$ admits a Néron lft-model over $S.$ \label{uninéronexistlem}
\end{lemma}
\begin{proof}
Let $S_B$ and $S_A$ denote the integral closures of $S$ in $B$ and $A,$ respectively. Because $A$ and $B$ are reduced and $S$ is excellent, $S_A$ and $S_B$ are (not necessarily connected) Dedekind schemes which are finite and locally free over $S,$ and we have obvious fppf-covers $S_B\to S_A \to S.$ Let $\mathscr{G}_{\mathrm{m}, S_B}$ be the Néron lft-model of ${\mathbf{G}_{\mathrm{m}, B}}$ over $S_B.$  By \cite{CGP}, Corollary A.5.4 (3), the sequence
$$0\to \Res_{A/K}{\mathbf{G}_{\mathrm{m}, A}}\to \Res_{B/K}{\mathbf{G}_{\mathrm{m}, B}} \to \Res_{A/K} (\Res_{B/A}{\mathbf{G}_{\mathrm{m}, B}}/{\mathbf{G}_{\mathrm{m}, A}})\to 0$$ is exact. By Lemma \ref{Gmclimprop}, the algebraic group $\Res_{B/A}{\mathbf{G}_{\mathrm{m}, B}}/{\mathbf{G}_{\mathrm{m}, A}}$ admits a Néron lft-model $\mathscr{R}$ over $S_A.$ By the discussion preceding Lemma \ref{Gmclimprop}, $\Res_{S_B/S} \mathscr{R}$ is the desired Néron lft-model of $(\Res_{B/K}{\mathbf{G}_{\mathrm{m}, B}})/\Res_{A/K}{\mathbf{G}_{\mathrm{m}, A}}$ over $S.$
\end{proof}

\subsection{Factorisation of birational morphisms of one-dimensional schemes}
We shall now proceed to showing that each finite birational morphism $f\colon X\to Y$ of reduced curves over an arbitrary field $\kappa$ can be written as a composition of push-outs along prime algebras $\kappa' \to A,$ where $\kappa'$ is a finite field extension of $\kappa.$ We shall set up the necessary technical framework regarding push-outs of schemes and seminormality in this section. This will be more general than immediately needed, since more powerful techniques will be required later. 

\subsubsection{Seminormality and seminormalisation}

Let us first recall a few definitions and results from \cite{Stacks}, Tag 0EUK:
\begin{definition}
Let $S$ be a scheme. We say that $S$ is \rm seminormal \it if for every open affine subscheme $U\subseteq S$ and all $x,y\in \Gamma(U, \Og_U)$ with $x^3=y^2$, there exists a unique $a\in \Gamma(U, \Og_U)$ such that $x=a^2$ and $y=a^3.$  \label{seminormaldef}
\end{definition}
Being seminormal is a local property of schemes by \cite{Stacks}, Tag 0EUP, i. e., it suffices to require the existence of one affine open cover of $S$ all of whose members have the property from the Definition above. It is easy to see that seminormal schemes are reduced. Moreover, given a scheme $S,$ there exists a \it seminormalisation \rm with some remarkable properties:

\begin{proposition}
Let $S$ be a scheme. Then there exists a seminormal scheme $S^{\mathrm{sn}}$ and a morphism $\varsigma\colon S^{\mathrm{sn}}\to S$ which is a universal homeomorphism (hence integral), induces isomorphisms on all residue fields, and satisfies the following universal property: for each universal homeomorphism $S'\to S$  which induces isomorphisms on all residue fields, the morphism $\varsigma\colon S^{\mathrm{sn}}\to S$ factors uniquely through $S'\to S.$ \label{seminormalisationexprop}
\end{proposition}
\begin{proof}
See \cite{Stacks}, Tag 0EUS (3).
\end{proof}\\
Now suppose that $S$ is a reduced Noetherian scheme. Let $\eta(S)$ be the disjoint union of the spectra of the fields of fractions of the (finitely many) irreducible components of $S.$ We let $\widetilde{S}\to S$ be the normalisation of $S$ in $\eta(S)$ and call $\widetilde{S}$ the \it normalisation of $S.$\rm \\
The following lemmata are certainly well-known; we include proofs for the sake of completeness: 
\begin{lemma}
Let $S$ be a scheme. Then the morphism $\varsigma\colon S^{\mathrm{sn}}\to S$ is an isomorphism if and only if $S$ is seminormal.
\end{lemma}
\begin{proof}
See \cite{Stacks}, Tag 0EUS (4).
\end{proof}
\begin{lemma}
Let $S$ be a scheme and let $U\to S$ be an open immersion. Then the canonical morphism $U^{\mathrm{sn}}\to U\times_SS^{\mathrm{sn}}$ is an isomorphism. The same is true if the morphism $U\to S$ is a localisation\footnote{i. e., for each affine open subscheme $V\subseteq S,$ the scheme $U\times_SV$ is affine and its ring of global sections is a localisation of $\Gamma(V, \Og_V)$ at a multiplicative subset.}. \label{loclemma}
\end{lemma}
\begin{proof}
By \cite{Stacks}, Tag 0EUP, the scheme $U\times_SS^{\mathrm{sn}}$ is seminormal if $U\to S$ is an open immersion, and an elementary calculation shows that this remains true after localisation. Hence it suffices to prove that the morphism $U\times_SS^{\mathrm{sn}} \to U$ is a universal homeomorphism which induces isomorphisms at all residue fields (\cite{Stacks}, Tag 0EUS (4)). But this is clear since both claims hold for the map $\varsigma\colon S^{\mathrm{sn}}\to S$ and are stable under localisation.
\end{proof}
\begin{lemma}
Let $T$ be a normal Noetherian scheme. Then $T$ is seminormal. Moreover, for any reduced Noetherian scheme $S,$ the canonical morphism $\widetilde{S}\to S$ factors through the map $\varsigma\colon S^{\mathrm{sn}}\to S.$ \label{snnormlemma}
\end{lemma}
\begin{proof}
For the first claim, we may assume without loss of generality that $T$ is affine and integral. Then $\Gamma(T, \Og_T)$ is an integral domain. Let $x,y\in \Gamma(T, \Og_T)$ such that $x^3=y^2.$ If $x=0,$ then $y=0$ and $x=0^2,$ $y=0^3.$ If $x\neq0,$ then $(y/x)^3 - y=0,$ so the element $y/x$ of $\mathrm{Frac}\, \Gamma(T, \Og_T)$ is integral over $\Gamma(T, \Og_T).$ Since $T$ is normal, this implies that $a:=y/x\in \Gamma(T, \Og_T).$ Hence $y=ax,$ which implies that $x^3=y^2=a^2x^2,$ so $x=a^2.$ This, in turn, implies that $y=ax=a^3.$ Therefore $T$ is seminormal. \\
For the second claim, we may once more assume that $S$ is affine. Let $M$ be the total ring of fractions of $\Gamma(S, \Og_S).$ By Lemma \ref{loclemma}, the morphism $\Spec M = \Spec M^{\mathrm{sn}} \to \Spec M\times_SS^{\mathrm{sn}}$ is an isomorphism. In particular, we obtain a morphism $\Gamma(S^{\mathrm{sn}}, \Og_{S^{\mathrm{sn}}})\to M\otimes_{\Gamma(S, \Og_S)} \Gamma(S^{\mathrm{sn}}, \Og_{S^{\mathrm{sn}}}) = M.$ Since the morphism $\varsigma\colon S^{\mathrm{sn}}\to S$ is a universal homeomorphism and therefore integral (\cite{Stacks}, Tag 04DF), we obtain our desired factorisation $\widetilde{S} \to S^{\mathrm{sn}} \to S$ of $\widetilde{S} \to S.$ 
\end{proof}
\begin{corollary}
Let $S$ be a reduced Noetherian scheme. Then both morphisms $\widetilde{\varsigma} \colon \widetilde{S} \to S^{\mathrm{sn}}$ and $\varsigma\colon S^{\mathrm{sn}}\to S$ are scheme-theoretically dominant, i. e., the canonical maps $\Og_S \to \varsigma_\ast \Og_{\mathrm{sn}}$ and $\Og_{S^{\mathrm{sn}}}\to \widetilde{\varsigma}_\ast \Og_{\widetilde{S}}$ of sheaves on the small Zariski (and étale) sites are injective. In particular, if the normalisation morphism $\nu\colon \widetilde{S}\to S$ is finite, then so are both $\varsigma$ and $\widetilde{\varsigma}.$ \label{schdomcor}
\end{corollary}
\begin{proof}
The normalisation morphism $\nu$ is scheme-theoretically dominant by construction. Since $\nu=\varsigma\circ\widetilde{\varsigma},$ we obtain a factorisation $\Og_S \to \varsigma_\ast \Og_{S^{\mathrm{sn}}}\to \nu_\ast \Og_{\widetilde{S}}.$ This immediately implies that the map $\Og_S \to \varsigma_\ast \Og_{S^{\mathrm{sn}}}$ is injective. We also see that the morphism $\varsigma_\ast \Og_{S^{\mathrm{sn}}}\to \varsigma_\ast \widetilde{\varsigma}_\ast \Og_{\widetilde{S}}$ is injective. If $\mathscr{F}$ denotes the kernel of the map $\Og_{S^{\mathrm{sn}}}\to \widetilde{\varsigma}_\ast \Og_{\widetilde{S}},$ then this implies that $\varsigma_\ast\mathscr{F}=0.$ Since $\varsigma$ is an homeomorphism, this implies that $\mathscr{F}=0,$ so $\Og_{S^{\mathrm{sn}}} \to \widetilde{\varsigma}_\ast \Og_{\widetilde{S}}$ is indeed injective.
\end{proof}
\subsubsection{Push-outs of schemes} \label{pushoputpara}
We shall now recall several results regarding push-outs (i. e., fibre coproducts) in the category of schemes. It is well-known that general push-outs of schemes need not exist. However, there are several important cases where push-outs do exist. They have been studied by Ferrand \cite{Fer} and (independently) by Schwede \cite{Schw}. The behaviour of push-outs under arbitrary base change has been studied by the author\footnote{See also \cite{Brion}, proof of Lemma 2.2, and the references therein.} \cite{Ov}, Paragraph 4, where it was shown that push-outs can be used to construct models of some singular curves over discrete valuation rings, and to study their Picard functors. We shall extend those methods to the extent necessary for our purposes. As  the language of \cite{Schw} was used in \cite{Ov}, we shall continue using \cite{Schw} as our reference for results on push-outs of schemes. Some similar results are also contained in \cite{Stacks}. Let us begin with the following results, which generalise \cite{Ov}, Proposition 4.0.2. \\
Throughout this Paragraph, we shall work with the following setup: let $f\colon X\to Y$ be a morphism of schemes. Let $t\colon T\to Y$ and $z\colon Z \to Y$ be schemes \it affine \rm over $Y,$ let $\xi \colon Z\to T$ be a morphism over $Y,$ and let $\widetilde{\iota}\colon Z\to X$ be a closed immersion. We summarise this in the following diagram:
$$\begin{tikzcd}
& X \arrow[swap, bend right]{ldd}{f} \\
& Z\arrow[swap]{ld}{z} \arrow[swap]{u}{\widetilde{\iota}} \arrow{r}{\xi} & T\arrow[bend left]{lld}{t} \\
Y
\end{tikzcd}.$$
Moreover, throughout this Paragraph, we shall impose the following \\
\\
$\mathbf{Standing  \,\, assumption}.$ Each point of $Y$ has an open affine neighbourhood $U$ such that the induced morphism $z^{-1}(U)\to f^{-1}(U)$ factors through an open affine subset of $f^{-1}(U).$

\begin{proposition} \rm (Cf. \cite{Fer}, Théorème 7.1 A)) \it 
Assume that the standing assumption is satisfied. Then the push-out $X\cup_ZT$ (taken in the category of ringed spaces) is a scheme. Moreover, the morphisms of ringed spaces $X \to X\cup_ZT$ and $T\to X\cup_ZT$ are morphisms of schemes, which turn $X\cup_ZT$ into a push-out in the category of schemes. \\
There is a canonical morphism $X\cup_ZT \to Y$ which is the push-out of the diagram
$$\begin{CD}
X\\
@AAA\\
Z@>>>T
\end{CD}$$
in the category of schemes over $Y.$ The map $T\to X\cup_ZT$ is a closed immersion and the morphism $X\backslash Z \to X\cup_ZT \backslash T$ is an isomorphism. \label{Schwede}
\end{proposition}
\begin{proof}
We may assume, without loss of generality, that $Y$ is affine; this follows from the fact that, for open subschemes $U\subseteq X$ and $W\subseteq T$ with common intersection $\Omega$ with $Z,$ the push-out $U\cup_{\Omega} W$  is an open subscheme of $X\cup_ZT$ by construction. Then $T$ and $Z$ are affine as well by assumption. Let $V$ be an open affine subset of $X$ through which $\widetilde{\iota}$ factors. By \cite{Schw}, Theorem 3.5, the push-out $V\cup_ZT$ exists in the category of schemes, and is isomorphic to $\Spec(\Gamma(V, \Og_V)\times_{\Gamma(Z, \Og_Z)} \Gamma(T, \Og_T)).$ Moreover, this scheme is the push-out of the relevant diagram in the category of ringed spaces. By \cite{Schw}, Theorem 3.4, the scheme $V\backslash Z$ is canonically an open subscheme of $V\cup_ZT \backslash T.$ Hence we can glue the schemes $X\backslash Z$ and $V\cup_ZT$ along $V\cup_ZT \backslash T.$ One now checks easily that the scheme thus constructed satisfies the universal property of the push-out in the category of schemes. The remaining claims can be proven in a purely formal manner, which will be left to the reader.
\end{proof}\\
The scheme $X\cup_ZT$ we just constructed fits into our setup as follows:

$$\begin{tikzcd}
& X \arrow[swap, bend right]{ldd}{f} \arrow{r}{\psi} & X\cup_ZT \\
& Z\arrow[swap]{ld}{z} \arrow[swap]{u}{\widetilde{\iota}} \arrow{r}{\xi} & T\arrow[bend left]{lld}{t} \arrow[swap]{u}{\iota}\\
Y
\end{tikzcd}.$$

We shall now introduce further assumptions on the morphisms $z,$ $t,$ and $\xi,$ and prove results about the push-outs we just constructed which would fail without these additional conditions. These assumptions will be referred to by the roman numerals (i),...,(v), and they will be \it cumulative, \rm i. e., once introduced, they will remain in place, as will the standing assumption. \\
\\
$\mathbf{Further  \,\, assumptions \, \, I}.$ In addition to the standing assumption, suppose the following: \\
(i) The scheme $Y$ is locally Noetherian, \\
(ii) the morphisms $f\colon X\to Y$ and $t\colon T\to Y$ are of finite type, and\\
(iii) the morphism $\xi \colon Z\to T$ is finite.

\begin{proposition} \rm (Cf. \cite{Brion}, 2.1) \it 
Suppose that the standing assumptions as well as (i), (ii), and (iii) are satisfied. Then the scheme $X\cup_ZT$ is of finite type over $Y.$ Moreover, if $X$ and $T$ are proper over $Y,$ then so is $X\cup_ZT.$ \label{pushfiniteProp}
\end{proposition}
\begin{proof}
We may assume that $Y$ is affine. By the construction of $X\cup_ZT$ from the proof of the previous Proposition, it suffices to show that $V\cup_ZT$ is of finite type. This follows from \cite{Stacks}, Tag 00IT, or the argument from the proof of \cite{Ov}, Proposition 4.0.2. Since the morphism $X\sqcup T \to X \cup_ZT$ is surjective (which follows from the topological construction \cite{Schw}, proof of Theorem 3.4) and $X\sqcup T$ is proper over $Y$ by assumption, we see as in the proof of \cite{Ov}, Proposition 4.0.2 that the morphism $X\cup_ZT\to Y$ is universally closed. It follows from the proof of \cite{Stacks}, Tag 00IT that the morphism $X\sqcup T \to X\cup_ZT$ is finite, and we already know that it is surjective. Hence \cite{Stacks}, Tag 09MQ, implies that the map $X\cup_ST \to Y$ is separated. Putting things together, we find that $X\cup_ZT$ is proper over $Y,$ as claimed. 
\end{proof}

\begin{proposition} \rm (Cf. \cite{Fer}, Théorème 7.1 A)) \it 
Assuming the standing assumption, as well as (i), (ii), and (iii), the following assertions hold: \\
(a) The map $X\to X\cup_ZT$ is finite, and\\
(b) the canonical morphism 
$Z\to X\times_{X\cup_ZT} T$ is an isomorphism.\label{Preimageprop}
\end{proposition}
\begin{proof}
As before, we may assume that $Y$ is affine, which implies that $Z$  and $T$ are affine as well. We may choose an open affine subset $V$ of $X$ through which the map $Z\to X$ factors. Claim (b) is then equivalent to the assertion that the map $Z\to V\times_{V\cup_ZT}T$ is an isomorphism. By Proposition \ref{Schwede}, all we must prove is that the map 
$$\Gamma(V, \Og_V)\otimes_{\Gamma(V, \Og_V)\times_{\Gamma(Z, \Og_Z)}\Gamma(T, \Og_T)}\Gamma(T, \Og_T) \to \Gamma(Z, \Og_Z)$$ is an isomorphism of rings. Since the map $\Gamma(V, \Og_V)\times_{\Gamma(Z, \Og_Z)}\Gamma(T, \Og_T)\to \Gamma(T, \Og_T)$ is surjective, every element of the tensor product above can be written as $\alpha\otimes 1$ for some $\alpha\in \Gamma(V, \Og_V).$ If this element vanishes in $\Gamma(Z, \Og_Z),$ then the same is true for $\alpha.$ This implies that $\alpha$ is the image of $(\alpha, 0)\in \Gamma(V, \Og_V)\times_{\Gamma(Z, \Og_Z)}\Gamma(T, \Og_T),$ which means that $\alpha\otimes 1=0.$ Hence the map above is injective; its surjectivity follows immediately from the fact that $Z\to V$ is a closed immersion. Claim (a) follows from the finiteness of $X\sqcup T \to X\cup_ZT,$ which we have already established in the proof of the preceding Proposition.  
\end{proof}\\
Having established the existence of push-outs under certain conditions, we shall now prove that, under appropriate flatness assumptions, push-outs commute with arbitrary base change, generalising Propositions 4.0.3, 4.0.4, and 4.0.5 from \cite{Ov} (see also the proof of Lemma 2.2 in \cite{Brion}). This will be used to study Picard functors by methods introduced in \cite{Ov}, which we shall generalise. It is not difficult to prove that push-outs commute with flat base change, which has already been observed by Ferrand \cite{Fer}, Lemme 4.4.\\
\\
We now impose the following \\
\\
$\mathbf{Further  \,\, assumptions \, \, II}.$ In addition to the standing assumption and (i), (ii), (iii), suppose the following: \\
(iv) The morphism $\xi\colon Z\to T$ is faithfully flat, and\\ 
(v) the cokernel of the injective map $\xi^\ast\colon t_\ast \Og_T\to z_\ast \Og_Z$ is projective locally in the Zariski topology on $Y$ (\cite{Stacks}, Tag 05JP).

\begin{proposition}
Assuming the standing assumption and (i),...,(v), the following holds: \\
(a) For any scheme $Y'\to Y,$ denote by $X',$ $Z',$ and $T'$ the base changes of $X,$ $Z,$ and $T$ to $Y',$ respectively. Then the morphism \label{Basechangeprop}
$$X'\cup_{Z'}T'\to(X\cup_ZT)\times_YY' $$
is an isomorphism. In particular, the map $X \to X\cup_ZT$ is scheme-theoretically dominant and remains so after any base change $Y'\to Y.$\\
(b) If $X$ is flat over $Y,$ then so is $X\cup_TZ.$
\end{proposition}
\begin{proof}
(a) We may assume that both $Y'$ and $Y$ are affine, and that the cokernel of the map $t_\ast \Og_T\to z_\ast \Og_Z$ is a projective quasi-coherent sheaf on $Y.$ As before, we choose an open affine subscheme $V$ of $X$ through which the map $Z\to X$ factors. We write $V':=V\times_YY'.$ It suffices to show that the morphism $V'\cup_{Z'}T'\to(V\cup_ZT)\times_YY'$ is an isomorphism; this follows from the fact that push-outs are local (see the proof of Proposition \ref{Schwede}). To simplify the notation, we shall write $R,$ $R',$ $A,$ $A',$ $B,$ $B',$ $C,$ and $C'$ for the rings of global sections of $Y,$ $Y',$ $V,$ $V',$ $Z,$ $Z',$ $T,$ and $T',$ respectively. In particular, we have $A'=A\otimes_RR',$ $B'=B\otimes_RR',$ and $C'=C\otimes_RR'.$ We must now prove that the canonical map 
$$(A\times_BC)\otimes_RR'\to A'\times_{B'}C'$$ is an isomorphism, which we shall do by adapting the proof of \cite{Ov}, Proposition 4.0.3. We begin by observing that the maps $C\to B$ and $C'\to B'$ are faithfully flat and hence injective. Now we consider the exact sequence
$$0\to A\times_BC\to A \to B/C\to 0.$$ By assumption, the $R$-module $B/C$ is projective, so this sequence splits. This implies, in particular, that it remains exact after arbitrary base change, and we obtain an exact sequence
$$0\to (A\times_BC)\otimes_RR'\to A' \to (B/C)\otimes_RR'\to 0.$$ The same argument shows that the exact sequence $0\to C\to B \to B/C \to 0$ remains exact after tensoring with $R'.$ Hence we obtain a canonical isomorphism $B'/C'\to (B/C)\otimes_RR'$ However, the kernel of the morphism $A'\to B'/C'$ is clearly the same as $A'\times_{B'}C',$ which proves our claim. \\
(b) This claim is local in the Zariski topology on both source and target, so we may again assume that $Y,$ $X,$ $Z,$ and $T$ are all affine. With the notation as in the proof of (i), we must show that if $A$ is flat over $R,$ then $\mathrm{Tor}^R_1(A\times_BC,-)=0.$ This follows immediately from the long exact sequence associated with $0 \to A\times_BC\to A \to B/C \to 0$ and  the fact that both $A$ and $B/C$ are flat over $R.$ 
\end{proof}\\
The following is a generalisation of \cite{Ov}, Proposition 4.0.4 (see also \cite{Fer}, Lemme 4.4 and \cite{Stacks}, Tag 0D2K). We include a proof since the result is crucial for later applications.
\begin{proposition} 
(a) Assuming the standing assumption and (i),...,(iii), the following holds: for any flat morphism $F\to X\cup_ZT$ of schemes, the canonical map 
$$\lambda\colon (F\times_{X\cup_ZT}X)\cup_{F\times_{X\cup_ZT}Z}(F\times_{X\cup_ZT}T)\to F$$ is an isomorphism. \\
(b) Moreover, under the standing assumption as well as (i),..., (v), part (a) remains true after arbitrary base change $Y'\to Y$ (i. e., even if $Y'$ is not locally Noetherian).  \label{BasechangepropII}
\end{proposition}
\begin{proof}
We may once again assume, without loss of generality, that $Y$ (and hence $Z$ and $T$) are affine. As before, we choose an open affine subscheme $V$ of $X$ through which $Z\to X$ factors. We may then replace $F$ by the pre-image of $V$ in $F$ and assume that the morphism $F\to X\cup_ZT$ factors through $V\cup_ZT.$ Moreover, we may assume that $F$ is affine. Both those claims follow from the fact that for any open affine $U\subseteq F,$ we have 
$$\lambda^{-1}(U)=(U\times_{X\cup_ZT}X)\cup_{U\times_{X\cup_ZT}Z}(U\times_{X\cup_ZT}T);$$ this is a consequence of the fact that the push-outs we consider are already push-outs in the category of ringed spaces. This allows us to translate the claim into a purely algebraic assertion: with the notation from the proof of Proposition \ref{Basechangeprop} and $D:=\Gamma(F, \Og_F),$ we must prove that the canonical morphism 
$$\lambda^\ast\colon D\to (D\otimes_{A\times_BC}A)\times_{D\otimes_{A\times_BC}B}(D\otimes_{A\times_BC}C)$$ is an isomorphism. This follows from \cite{Stacks}, Tag 08KQ. We give a sightly different proof which is an adaption of the proof of \cite{Ov}, Proposition 4.0.4. We begin by observing that the map $D\otimes_{A\times_BC}C\to D\otimes_{A\times_BC}B$ is injective because it is faithfully flat. Hence the target of $\lambda^\ast$ is equal to the set of all elements of $\delta\in D\otimes_{A\times_BC}A$ whose image in $D\otimes_{A\times_BC}B$ comes from $D\otimes_{A\times_BC}C.$ Since the map $A\times_BC \to C$ is surjective, every element of $D\otimes_{A\times_BC}C$ is an elementary tensor. Let $\delta$ be an element of the target of $\lambda^\ast$ and let $\overline{\delta}$ bet its image in $D\otimes_{A\times_BC}B.$ Then we can find an element $d\in D$ such that $\overline{\delta}=d\otimes 1$ in $D\otimes _{A\times_BC}B.$
Let $I:=\ker(A\to B).$ Because the sequence 
$$D\otimes_{A\times_BC}I\to D\otimes_{A\times_BC}A\to D\otimes_{A\times_BC}B\to 0$$ is exact, we can find elements $\eta_1,..., \eta_r \in I$ (for some $r\in \N$) such that 
$$\delta - d\otimes 1 = \sum_j d'_j\otimes \eta_j$$ in $D\otimes_{A\times_BC}A$ for appropriately chosen elements $d'_j\in D.$ However, since $I\subseteq A\times_BC,$ there exists $d'\in D$ with the property that $\sum_j d'_j\otimes \eta_j = d'\otimes 1.$ This shows that $\delta= (d+d')\otimes 1,$ so that $\lambda^\ast$ is surjective. Because $D$ is flat over $A\times_BC,$ the map $D\to D\otimes_{A\times_BC}A$ is injective, which shows that $\lambda^\ast$ is injective as well. Finally, note that the Noetherian hypothesis was only used in Proposition \ref{Schwede} in order to prove that the push-out is of finite type over $Y,$ which we have not used in this proof. Hence the final claim follows from Proposition \ref{Basechangeprop}.
\end{proof}\\
\\
For later use, we shall at this point study line bundles on the push-out $X\cup_ZT$ in terms of line bundles on $X,$ $Z,$ and $T.$ This is inspired by \cite{Stacks}, Tag 0D2G. Once again, we keep the notation and assumptions from above (i. e., the standing assumption and (i),...,(v)). Let $\xi$ and $\widetilde{\iota}$ be the morphisms already used in the setup at the beginning of this paragraph. We define a category $\mathcal{C}$ as follows: the objects are triples $(\mathscr{M}, \mathscr{N}, \lambda),$ where $\mathscr{M}$ and $\mathscr{N}$ are line bundles on $X$ and $T,$ respectively, and $\lambda \colon \widetilde{\iota}^\ast \mathscr{M}\to \xi^\ast\mathscr{N}$ is an isomorphism. A morphism $(\mathscr{M}, \mathscr{N}, \lambda)\to (\mathscr{M}', \mathscr{N}', \mathscr{\lambda'})$ in $\mathcal{C}$ is a pair $(\alpha, \beta)$ consisting of morphisms\footnote{By a \it morphism of line bundles \rm we mean a morphism of quasi-coherent sheaves.} $\alpha\colon \mathscr{M}\to \mathscr{M}'$ and $\beta\colon \mathscr{N}\to \mathscr{N'}$ such that $\lambda'\circ \widetilde{\iota}^\ast\alpha=\xi^\ast\beta\circ \lambda.$ (The category $\mathcal{C}$ is a \it fibre product of categories; \rm see \cite{Stacks}, Tag 003R.) Then we have 
\begin{proposition} \rm (Cf. \cite{Brion}, section 2.2) \it
Let $\mathcal{P}$ denote the category of line bundles on $X\cup_ZT.$ Then the functor $\mathcal{P}\to \mathcal{C}$ given by
$$\mathscr{L}\mapsto (x ^ \ast\mathscr{L}, y^\ast \mathscr{L}, \lambda_{\mathscr{L}})$$ is an equivlence of categories. Here, $x\colon X\to X\cup_ZT$ and $y\colon T\to X\cup_ZT$ denote the canonical morphisms, and 
$$\lambda_{\mathscr{L}}\colon \widetilde{\iota}^\ast x^\ast \mathscr{L}\to \xi^\ast y^\ast \mathscr{L}$$ denotes the canonical isomorphism. \label{linebundlesequivprop}
\end{proposition}
\begin{proof}
We may assume without loss of generality that $X$ and $Y$ are affine. Hence \cite{Stacks}, Tag 0D2J tells us that the claim is true if we replace line bundles by finite locally free modules. We must therefore prove that the equivalence of categories from \it loc. cit. \rm translates line bundles to line bundles. Clearly, if $\mathscr{L}$ is a line bundle then so are $x^\ast \mathscr{L}$ and $y^\ast \mathscr{L}.$ On the other hand, suppose $\mathscr{F}$ is a locally free coherent sheaf on $X\cup_ZT$ such that both $x^\ast \mathscr{F}$ and $y^\ast \mathscr{F}$ are line bundles. Then $\mathscr{F}$ is a line bundle because the map $X\sqcup T\to X\cup_ZT$ is surjective, so the rank of $\mathscr{F}$ is equal to 1 everywhere. 
\end{proof}

\subsubsection{The factorisation theorem for curves} \label{factoristionpara}
In this paragraph, we shall prove a factorisation theorem for finite dominant birational morphisms of reduced curves. (Throughout this paragraph, a \it reduced curve \rm will mean a reduced purely one-dimensional schemes of finite type over a field. A \it birational morphism \rm of reduced curves is a morphism which induces an isomorphism of dense open subsets of source and target.) Our result will generalise \cite{Stacks}, Tag 0C1L, where the factorisation theorem is proven over algebraically closed fields. The proof given in \it loc. cit. \rm can be taken with some minor modifications. 
First recall that a commutative diagram
$$\begin{CD}
X@>>>X'\\
@AAA@AAA\\
Z@>>>T
\end{CD}$$
of schemes is \it co-Cartesian \rm if $X'$ satisfies the universal property of the push-out in the category of schemes. 
\begin{proposition}
Let $\kappa$ be an arbitrary field and let $X$ and $X'$ be reduced curves over $\kappa.$ Moreover, let $\beta\colon X\to X'$ be a finite birational morphism (in particular, the canonical map $\Og_{X'}\to \beta_\ast \Og_X$ is injective, i. e., $\beta$ is scheme-theoretically dominant). Assume that, for every factorisation $X\to X''\to X',$ if both morphisms therein appearing are scheme-theoretically dominant, then at least one of them is an isomorphism. Then either $\beta$ is an isomorphism, or there exists a closed point $x'\in X'$ such that the following assertions hold: \\
(i) The scheme $X\times_{X'}\Spec \kappa(x')$ is isomorphic to $\Spec A,$ where $A$ is a prime algebra over $\kappa(x'),$\\
(ii) if $U$ is the complement of $x'$ in $X',$ the induced morphism $\beta^{-1}(U)\to U$ is an isomorphism. \\
(iii) Let $\alpha \colon \Spec A\to \Spec \kappa(x')$ be the canonical morphism. Then the diagram \label{FactorizationI}
$$\begin{CD}
X@>{\beta}>>X'\\
@AAA@AA{j}A\\
\Spec A@>>{\alpha}>\Spec \kappa(x')
\end{CD}$$
is co-Cartesian.
\end{proposition}
\begin{proof} We proceed as in the proof of \cite{Stacks}, Tag 0C1L. Consider the cokernel $\mathcal{Q}$ of $\Og_{X'}\to \beta_\ast\Og_{X}.$ Then we have $\mathcal{Q}=\mathcal{Q}_1\oplus...\oplus \mathcal{Q}_r$ for some $r\in \N$, where each $\mathcal{Q}_j$ is topologically supported on a closed point $x'_j$ of $X'$ and non-zero. This follows from the assumption that $\beta$ is birational. If $r>1,$ then the $\Og_{X'}$-algebra $\beta_\ast \Og_{X}$ has a proper subalgebra, contradicting our assumption on $\beta.$ Hence we must have $r=0$ (in which case $\beta$ is an isomorphism), or $r=1.$ We shall now prove that, in the latter case, assertions (i), (ii), and (iii) are satisfied. Claim (ii) is immediately clear. To see claim (i), we consider the morphism $\beta_\ast \Og_{X}\to j_\ast A,$ where we denote by $A$ the ring of global functions of the affine scheme $X\times_{X'}\Spec \kappa(x'),$ viewed as a sheaf on $\Spec \kappa(x').$ If $A$ had a proper $\kappa(x')$-subalgebra, then its pre-image in $\beta_\ast \Og_X$ would again give rise to a non-trivial factorisation of $X\to X'.$ To see that $\kappa(x') \to A$ is not an isomorphism, let $B$ be the ring of global functions of the affine scheme $X\times_{X'} \Spec \Og_{X', x'},$  and let $\mathfrak{m}$ be the maximal ideal of $\Og_{X', x'}.$ Consider the exact sequence of $\Og_{X', x'}$-modules
$$0 \to \Og_{X', x'} \to B \to Q \to 0,$$ where $Q$ is the obvious cokernel. If $\kappa(x') \to A$ were an isomorphism, we would have $Q/\mathfrak{m}Q=0,$ so Nakayama's lemma would imply that $\beta$ is an isomorphism. Hence $A$ is indeed a prime algebra over $\kappa(x'),$ which proves (i). All that now remains to be shown is claim (iii). 
To prove this assertion, note that we have a factorisation 
$$X\to X\cup_{\Spec A} \Spec \kappa(x') \to X'$$ of $\beta.$ Because $\kappa(x')\to A$ is not an isomorphism, neither is the first map in this factorisation. Hence our assumptions on $\beta$ imply that the second map is an isomorphism.
\end{proof}
\begin{corollary}
Let $X$ and $X'$ be reduced curves over a field $\kappa.$ Let $\beta\colon X\to X'$ be a finite birational morphism. If $\beta$ is not an isomorphism, then $\beta$ can be written as a composition 
$$X=X_1\overset{\beta_1}{\to}... \overset{\beta_{n-1}}{\to} X_n=X'$$ for some $n\in \N$ 
of morphisms of $\kappa$-schemes such that, for each $i=1,..., n-1,$ there exists a point $x_{i+1}$ in $X_{i+1},$ a prime $\kappa(x_{i+1})$-algebra $A_{i},$ and a closed immersion $\Spec A_{i}\to X_i$ with the property that the diagram
$$\begin{CD}
X_i@>{\beta_i}>>X_{i+1}\\
@AAA@AAA\\
\Spec A_i@>>> \Spec \kappa(x_{i+1})
\end{CD}$$
is co-Cartesian. \label{FactorizationII}
\end{corollary}
\begin{proof}
We shall once more adapt the proof of \cite{Stacks}, Tag 0C1L. By our assumptions on $\beta,$ we know that the cokernel of $\Og_{X'}\to \beta_\ast\Og_X$ is of finite length. We shall argue by induction on the length of $\mathcal{Q}.$  If $\mathrm{length}\, \mathcal{Q}=0,$ then $\beta$ is an isomorphism. If there is no proper subalgebra $\Og_{X'} \subseteq \mathscr{A} \subseteq \beta_\ast \Og_X,$ then the result follows from Proposition \ref{FactorizationI}. On the other hand, if such a subalgebra does exist, we can factor $\beta$ as 
$$X\to \boldsymbol{\mathrm{Spec}}\, \mathscr{A} \to X'.$$ Since the length of the cokernels of the induced maps on structure sheaves is strictly smaller than $\mathrm{length}\, \mathcal{Q}$ for both $ X\to \boldsymbol{\mathrm{Spec}}\, \mathscr{A}$ and $ \boldsymbol{\mathrm{Spec}}\, \mathscr{A} \to X',$ the result follows.
\end{proof}\\
We shall now apply this result to the normalisation morphism $\nu\colon\widetilde{X} \to X$ of a reduced curve $X$ over the field $\kappa.$ From Lemma \ref{snnormlemma}, we already know that we can factor $\nu$ as $\widetilde{X}\overset{\widetilde{\varsigma}}{\to} X^{\mathrm{sn}}\overset{\varsigma}{\to} X,$ where $X^{\mathrm{sn}}$ denotes the seminormalisation of $X.$ A result very similar to part (ii) of the following Theorem has previously been obtained by Laurent (\cite{Lau}, Lemmata 3.1(c) and 3.7), who uses the language of \cite{Fer}. Before we state the next Theorem, we recall that, if $\kappa$ is a field and $X$ is a scheme of finite type over $\kappa,$ then the set of singular points (i. e., the set $\{x\in X\colon \text{$\Og_{X,x}$ is not a regular local ring}\}$) is a closed subset of $X.$ This follows from the fact that $X$ is excellent (\cite{Stacks}, Tag 07QU; excellent schemes have closed singular locus by definition).
\begin{theorem} \rm (Factorisation theorem) \it
Let $\kappa$ be an arbitrary field and let $C$ be a reduced curve over $\kappa.$ Denote by $\widetilde{C}$ the normalisation of $C.$ Then the following two assertions hold:\\
(i) If $C$ is not seminormal, the morphism $\varsigma\colon C^{\mathrm{sn}}\to C$ can be written as a composition \label{FactorisationTheorem}
$$C^{\mathrm{sn}}=C_1 \overset{\varsigma_1}{\to}... \overset{\varsigma_{n-1}}{\to} C_{n}=C$$ for some $n\in \N,$ such that, for each $i=1,..., n-1,$ there is a closed point $x_{i+1}$ in $C_{i+1}$ and a closed immersion $\Spec \kappa(x_{i+1})[\epsilon]/\langle \epsilon^2 \rangle \to C_{i},$ such that the diagram
$$\begin{CD}
C_i@>{\varsigma_i}>>C_{i+1}\\
@AAA@AAA\\
\Spec \kappa(x_{i+1})[\epsilon]/\langle \epsilon^2 \rangle @>>> \Spec \kappa(x_{i+1})
\end{CD}$$
is co-Cartesian.\\
(ii) Suppose $C$ is seminormal and let $C^{\mathrm{sing}}\subseteq C$ denote the set of non-regular points of $C$ (which is closed in $C$ and finite over $\kappa$) endowed with its reduced subscheme structure. Let $D\to C$ be a finite birational morphism of reduced curves. Then the morphism $D \to C$ has reduced fibres and the diagram
$$\begin{CD}
D@>>> C^{\mathrm{sn}}\\
@AAA@AAA\\
D\times_{C} C^{\mathrm{sing}} @>>> C^{\mathrm{sing}}
\end{CD}$$
is co-Cartesian.  
\end{theorem}
\begin{proof}
(i) Let $C^{\mathrm{sn}}=C_1 \overset{\varsigma_1}{\to}... \overset{\varsigma_{n-1}}{\to} C_{n}=C$ be the factorisation of $\varsigma$ from Corollary \ref{FactorizationII} (which applies because of Corollary \ref{schdomcor}). In the terminology of Corollary \ref{FactorizationII}, we must show that, for all $i=1,..., n-1,$ the prime algebra $A_i$ over $\kappa(x_{i+1})$ is isomorphic to $\kappa(x_{i+1})[\epsilon]/\langle \epsilon^2 \rangle.$ By construction, each of the morphisms $\varsigma_j$ is surjective. Because their composition is injective, we see that all $\varsigma_j$ are, in fact, bijective. In particular, there is no $i=1,... n-1$ with the property that $A_i\cong \kappa(x_{i+1})\times \kappa(x_{i+1}).$ Moreover, each morphism $\varsigma_j$ induces isomorphisms on all residue fields (which follows from the fact that this is true for $\varsigma$), so there can be no $i=1,..., n-1$ such that $A_i$ is a proper field extension of $\kappa(x_{i+1}).$ Hence the claim follows from Corollary \ref{FactorizationII} and Proposition \ref{Primeclassprop}. \\
(ii) We begin by showing that the morphism $D\to C$ has reduced fibres. Let $x\in C$ be a closed point. Since the claim is local in the Zariski topology on $C,$ we may assume that $C$ (and hence also $D$) are affine. Shrinking $C$ further if necessary, we may assume that $C$ is regular away from $x,$ so that $D\to C$ is an isomorphism away from $x.$ Let $f\in \Gamma(D, \Og_D)$ be an element such that $f\otimes 1$ is nilpotent in $\Gamma(D, \Og_D)\otimes_{\Gamma(C, \Og_C)} \kappa(x)$ (note that every element in this tensor product is an elementary tensor since $\Gamma(C, \Og_C)\to \kappa(x)$ is surjective). We obtain ring extensions
$$\Gamma(C, \Og_C)\subseteq \Gamma(C, \Og_C)[f] \subseteq \Gamma(D, \Og_D).$$ The morphisms $D \to \Spec \Gamma(C, \Og_C)[f]$ and $\Spec \Gamma(C, \Og_C)[f] \to C$ are clearly integral and hence surjective (\cite{Stacks}, Tag 00GQ). In particular, the morphism
$$ D\times_C \Spec \kappa(x) \to (\Spec \Gamma(C, \Og_C)[f]) \times_{C}\Spec \kappa(x)$$ is surjective, so $f\otimes 1$ is nilpotent as an element of $\Gamma(C, \Og_C)[f] \otimes_{\Gamma(C, \Og_C)}\kappa(x).$ Since this element generates $\Gamma(C, \Og_C)[f] \otimes_{\Gamma(C, \Og_C)}\kappa(x)$ over $\kappa(x),$ it follows that the morphism $\Spec \Gamma(C, \Og_C)[f] \to C$ is injective.
Suppose $y$ is the point of $\Spec \Gamma(C, \Og_C)[f]$ which is mapped to $x.$ By construction, $\kappa(y)$ is generated by the image of $f$ over $\kappa(x).$ However, since $f\otimes 1 \in \Gamma(C, \Og_C)[f]\otimes_{\Gamma(C, \Og_C)} \kappa(x)$ is nilpotent, the image of $f$ in $\kappa(y)$ vanishes and the map $\kappa(x)\to \kappa(y)$ is an isomorphism. Now \cite{Stacks}, Tags 04DF and 01S4 show that the morphism $\Spec \Gamma(C, \Og_C)[f] \to C$ is a universal homeomorphism. Because $C$ is seminormal, this morphism is, in fact, an isomorphism by Proposition \ref{seminormalisationexprop}, so $f\otimes1 \in \Gamma(D, \Og_D)\otimes_{\Gamma(C, \Og_C)} \kappa(x)$ vanishes.\\
To see that the diagram in (ii) is co-Cartesian, let $Z:=D\times_CC^{\mathrm{sing}}$ and consider the canonical map 
$$D\cup_Z C^{\mathrm{sing}} \to C.$$ This map is clearly surjective, and the topological description of the push-out (Proposition \ref{Schwede}) shows that it is injective. Moreover, because the morphism $D\backslash Z \to D\cup_Z C^{\mathrm{sing}}\backslash C^{\mathrm{sing}}$ is an isomorphism and the map $C^{\mathrm{sing}} \to D\cup_Z C^{\mathrm{sing}}$ is a closed immersion (Proposition \ref{Schwede} again), the morphism $D\cup_Z C^{\mathrm{sing}} \to C$ induces isomorphisms on all residue fields. Finally, tis canonical map is finite (Proposition \ref{Preimageprop}), so we can argue exactly as in the preceding step to show that the canonical morphism is an isomorphism.  
\end{proof}
\begin{lemma}
The factorisation of $\nu=\varsigma\circ \widetilde{\varsigma}$ given above Theorem \ref{FactorisationTheorem} commutes with (not necessarily finite) separable algebraic extensions of $\kappa.$ \label{snbasechangelem}
\end{lemma}
\begin{proof}
Because the morphism $\widetilde{C}\to C$ is finite, we can easily check that $C^{\mathrm{sn}}$ is the \it semi normalisation of \rm  $C$ \it in \rm $\widetilde{C}$ in the sense of \cite{Kol}, Definition 7.2.1. In particular, the formation of $C^{\mathrm{sn}}$ commutes with separable field extensions by \cite{Kol}, Proposition 7.2.6. Moreover, the formation of $\widetilde{C}$ commutes with separable algebraic extensions of $\kappa,$ which implies the Lemma.
\end{proof}
\subsection{The two Picard functors}\label{twopicsubsec}
Let $f\colon X\to Y$ be a morphism of schemes. As usual, a \it Picard functor \rm will be the sheafification of the presheaf
$$T\mapsto \Pic (T\times_YX),$$
with respect to a suitable Grothendieck topology on the category of schemes over $Y.$ The only topologies we shall use are the étale- and fppf-topologies. Following \cite{Kle}, Definition 9.2.2, we introduce two different Picard functors: 
\begin{definition}
Let $f\colon X\to Y$ be as above. We let $\Pic_{X/Y, \et}$ and $\Pic_{X/Y, \fppf}$ be the sheafification of the functor $T\mapsto  \Pic (T\times_YX)$ in the étale and the fppf-topology, respectively. If $\Pic_{X/Y, \fppf}$ is representable by an algebraic space, we shall refer to the $Y$-algebraic space representing it as $\Pic_{X/Y}.$ 
\end{definition}
We shall see later that we must work with $\Pic_{X/Y, \et}$ in an essential way, whereas most representability results for Picard functors are only available for $\Pic_{X/Y, \fppf}.$ This explains why we must introduce, and work with, both functors. We shall only ever consider very special morphisms $f\colon X\to Y.$ More precisely, we shall only apply Picard functors to relative curves $f\colon \mathscr{C}\to S,$ where $S$ is a Dedekind scheme. There is one general situation where the two Picard functors are isomorphic, and we shall make much use of this fact. First recall that a morphism $f\colon X\to Y$ is said to be \it cohomologically flat in dimension zero \rm if, for all morphisms $\phi\colon T\to Y,$ the canonical map
$$\phi^\ast f_\ast \Og_{X}\to ( \Id_T\times f)_\ast \Og_{T\times_YX}$$ is an isomorphism (there are different definitions in the literature; this is the one used in \cite{Art}, \cite{BLR}, \cite{Kle}, \cite{Liu}, and \cite{Ov}). Also recall that there is a canonical morphism 
$$\Pic_{X/Y, \et}\to \Pic_{X/Y, \fppf},$$ which comes from the fact that the fppf-topology is finer than the étale topology. The following result is known (\cite{Ray}, p. 28, (1.2) or \cite{BLR}, p. 203). However, since its proof is only sketched briefly in both of those sources, we give a complete proof here.
\begin{proposition}
Let $f\colon X\to Y$ be a proper morphism of schemes, with $Y$ not necessarily locally Noetherian. Assume moreover that $f$ is finitely presented and flat. Then the canonical map $\Pic_{X/Y, \et}\to \Pic_{X/Y, \fppf}$ is an isomorphism. \label{Picisomprop}
\end{proposition}
\begin{proof}
Everything in this proof must be shown after base change $T\to Y.$ To simplify the notation, we shall assume that $Y=T.$ By Grothendieck's theorem comparing étale and fppf-cohomology of smooth group schemes (\cite{Dix}, Théorème 11.7), the maps $H^ i _{\et}(X, \mathbf{G}_{\mathrm{m}, X}) \to H^ i _{\fppf}(X, \mathbf{G}_{\mathrm{m}, X})$ are isomorphisms for all $i\geq 0.$ We shall now show that the same is true for the morphisms
\begin{align} H^ i_{\et}(Y,  f_\ast \mathbf{G}_{\mathrm{m}, X}) \to H^ i_{\fppf}(Y,  f_\ast \mathbf{G}_{\mathrm{m}, X}).\tag{$\ast$}\label{stilltrueX}\end{align} This will occupy most of this proof and we proceed in several steps:\\
\it Step 1: \rm By \cite{BLR}, Chapter 8.1, Corollary 8 together with \cite{Stacks}, Tag 00DO, the sheaf $f_\ast \Og_X$ (considered as a sheaf on the big fppf-site of $Y$) is representable by an affine ring scheme $V$ of finite presentation over $Y,$ so $f_\ast \mathbf{G}_{\mathrm{m}, X}$ is a group scheme of finite presentation over $Y$ (\cite{BLR}, Chapter 8.1, Lemma 10). Hence\footnote{The step described in the preceding sentence is necessary because, in \cite{Dix}, Grothendieck uses a restricted fppf-site, whose objects are schemes \it locally of finite presentation \rm over the base scheme. See \cite{Dix}, p. 124.}, by \cite{Dix}, Lemme 11.1, we may assume that $Y=\Spec R_0$ for some strictly local (i. e., local and strictly Henselian) ring $R_0$ and must then show that $H^ i _{\fppf}(Y, f_\ast \mathbf{G}_{\mathrm{m},X})=0$ for all $i>0.$ This will be achieved in Step 3 below; we need one additional auxiliary fact:\\
\it Step 2: \rm Because $f$ is proper, the morphism $\tau\colon Y':= \boldsymbol{\Spec} \, f_\ast \Og_X \to Y$ is integral (\cite{Stacks}, Tag 03H2; it need not be finite since $R_0$ need not be Noetherian). We claim that the ring $R':=\Gamma(Y, f_\ast \Og_X)$ is a finite product of strictly local rings. Let $R$ be the image of the map $R_0\to R'$ and let $X_1,..., X_m$ be the connected components of the special fibre of $f\colon X\to Y.$ Then $R$ is strictly Henselian local (this is clear from the definitions). Moreover, let $R \subseteq F \subseteq R'$ be a finite $R$-subalgebra. Since the induced map $X\to \Spec F$ is surjective (\cite{Stacks}, Tags 03GY and 00GQ), we can write $F=\prod_{\sigma\in \Sigma_F}F_{\sigma}$ for a finite set $\Sigma_F$ such that there is a surjective map $\varpi_F\colon \{1,..., m\} \to \Sigma_F$, $F_\sigma$ is strictly local for all $\sigma\in \Sigma_F,$ and such that $X_j$ maps to the special point of $F_{\varpi_F(j)}$ for all $j=1,...,m$ (\cite{Stacks}, Tag 04GH (1)). Now choose a finite $R$-subalgebra $F_0$ of $R'$ such that $\Sigma_{F_0}$ has maximal cardinality among all such finite subalgebras. If $F_0\subseteq F$ is another finite $R$-subalgebra of $R',$ we obtain a unique surjective map $\Sigma_{F}\to \Sigma_{F_0}$ which commutes with the maps $\varpi_{F_0}$ and $\varpi_F.$ By our choice of $F_0,$ this map must be bijective, so we can identify $\Sigma_F$ and $\Sigma_{F_0}.$ Since $R'$ is the colimit of its finite $R$-subalgebras (\cite{Stacks}, Tag 02JJ), we obtain a decomposition
$$R'=\prod_{\sigma\in \Sigma_{F_0}}\varinjlim_{F_0\subseteq F} F_{\sigma}$$ (\cite{Stacks}, Tag 002W), where $F$ runs through the finite $R$-subalgebras of $R'$ containing $F_0.$ Each direct factor in this decomposition is the colimit of strictly local rings along finite local transition maps, and therefore visibly strictly local.\\
\it Step 3: \rm Let $f'\colon X\to Y'$ be the morphism induced by $f\colon X\to Y.$ Moreover, let $\mathcal{B}$ be the set of all schemes $U\to Y$ finite and locally free over $Y$ and let $\mathrm{Cov}$ be the set of all fppf-covers $\{U_\mu\to U\}_{\mu\in I}$ such that $U\in \mathcal{B}$ and such that all maps $U_\mu\to U$ are finite and locally free. By \cite{Stacks}, Tag 05WN, every fppf-covering of an element $U$ of $\mathcal{B}$ can be refined by an fppf-covering which consists only of quasi-finite affine morphisms, and since $U$ is the disjoint union of finitely many strictly local schemes (\cite{Stacks}, Tag 04GH (1)), the covering can be further refined by an element of $\mathrm{Cov}$ (\cite{Stacks}, Tag 04GG (13)). We shall now apply Cartan's criterion (\cite{Stacks}, Tag 03F9). Condition (1) in \it loc. cit. \rm is clearly verified, and we have just shown that so is condition (2). Hence we must now show that, for all $i> 0$ and all $\mathcal{U} \in \mathrm{Cov},$ we have
$\widecheck{H}^ i (\mathcal{U}, f_{\ast} \mathbf{G}_{\mathrm{m}, X}) = 0.$ Pick an element $\mathcal{U} = \{U_\mu\to U\}_{\mu\in I} \in \mathrm{Cov}$ and let $\mathcal{U}'$ be the fppf-cover $\{ U_\mu\times_YY' \to U\times_YY'\}_{\mu\in I}.$ Then, for any $i>0,$
\begin{align*}
\widecheck{H}^ i (\mathcal{U}, f_{\ast} \mathbf{G}_{\mathrm{m}, X}) &= \widecheck{H}^ i (\mathcal{U}', f'_\ast \mathbf{G}_{\mathrm{m}, X}) = \widecheck{H}^ i (\mathcal{U}',\mathbf{G}_{\mathrm{m}, Y'});
\end{align*}
the first equality is true by construction, and the second because taking the pushforward of the structure sheaf always commutes with \it flat \rm base change (\cite{Stacks}, Tag 02KH).
Since, for all $p\in \N_0$ and $\mu_0,..., \mu_p\in I,$ the scheme $U_{\mu_0}\times_U ... \times_U U_{\mu_p}\times_YY'$ is the finite disjoint union of strictly local schemes (\cite{Stacks}, Tag 04GH (1)), it follows from \cite{Dix}, Théorème 11.7 that $$H_{\fppf}^ i (U_{\mu_0}\times_U ... \times_U U_{\mu_p}\times_YY', \mathbf{G}_{\mathrm{m}, U_{\mu_0}\times_U ... \times_U U_{\mu_p}\times_YY'})=0$$ for all $i>0.$ However, \cite{Stacks}, Tags 03F7 and 04GH (1) together with \cite{Dix}, Théorème 11.7 now show that 
\begin{align*}
\widecheck{H}^ i (\mathcal{U}', \mathbf{G}_{\mathrm{m}, Y'}) &= H_{\fppf}^ i ( U\times_YY', \mathbf{G}_{\mathrm{m}, U\times_YY'}) \\
&=0.
\end{align*}
In particular, we may apply \cite{Stacks}, Tag 03F9 and conclude that $H^ i_{\fppf}(Y, f_\ast\mathbf{G}_{\mathrm{m}, X})=0$ for all $i>0.$ Hence the morphisms (\ref{stilltrueX}) are indeed isomorphisms and the Proposition follows from a standard argument involving the Leray spectral sequence and the lemma of five homomorphisms; see \cite{BLR}, p. 203, \cite{Kle}, p. 257, or \cite{Ov}, p. 6462. 
\end{proof}\\
The following general result is due to M. Artin \cite{Art}:
\begin{theorem} Let $Y$ be a locally Noetherian scheme. Let $f\colon X\to Y$ be a proper and flat morphism which is cohomologically flat in dimension zero. Then $\Pic_{X/Y, \fppf}$ is representable by an algebraic space\footnote{Recall that all algebraic spaces in this article will be quasi-separated and locally separated, i.e., the diagonal will be a quasi-compact immersion.} locally of finite presentation over $Y.$ \label{representthm}
\end{theorem}
\begin{proof}
The Noetherian assumption on $Y$ guarantees that $f$ is of finite presentation. A proof of the Theorem above is presented in \cite{Art}, Theorem 7.3, with a small correction given in the Appendix to \cite{ArtII}. 
\end{proof}\\
The reader should bear in mind that $\Pic_{X/Y}$ need not be smooth over $Y,$ even if $Y$ is the spectrum of an algebraically closed field. However, we have the following
\begin{proposition}
Let $f\colon X\to Y$ be as in the preceding Theorem. Suppose that, for all $y\in Y,$ we have $H^2(X_y, \Og_{X_y})=0,$ where $X_y:=X\times_Y\Spec \kappa(y).$ Then $\Pic_{X/Y}$ is smooth over $Y.$ \label{picsmoothprop}
\end{proposition}
\begin{proof}
By Proposition \ref{Picisomprop}, we know that $\Pic_{X/Y, \et}\cong \Pic_{X/Y, \fppf}$. Hence the claim follows from \cite{Kle}, Proposition 9.5.19. 
\end{proof}\\
As a next step, we shall show that the étale Picard functor interacts very well with the push-out construction, thereby generalising \cite{Ov}, Lemma 6.0.1. This is the place at which we must use the étale Picard functor; the proof given below would not work in the fppf-topology\footnote{More precisely, we shall use the fact that if $f\colon X\to Y$ is a finite morphism of schemes, then the functor $f_\ast-$ is exact in the étale topology. At present, the analogous statement is not known for the fppf-topology even if $f$ is a closed immersion, cf. \cite{Stacks}, Tag 04C5.}. Let $f\colon X\to Y$ be proper and flat. Moreover, let $Z\overset{\xi}{\to} T\overset{t}{\to} Y$ be as in the setup described at the beginning of Paragraph \ref{pushoputpara}, and suppose moreover that the standing assumption as well as the assumptions (i),...,(v) listed there are satisfied. Let $X':=X\cup_ZT$ and let $f'\colon X'\to Y$ be the structure morphism.
\begin{proposition}
Let $\iota\colon T\to X'$ be the canonical closed immersion (see Proposition \ref{Schwede}) and assume that $t$ is a finite morphism. Then we have an exact sequence 
\begin{align*}
0\to f'_\ast {\mathbf{G}_{\mathrm{m}, X'}} \to f_\ast {\mathbf{G}_{\mathrm{m}, X}} \to t_\ast ((\Res_{Z/T}{\mathbf{G}_{\mathrm{m}, Z}})/{\mathbf{G}_{\mathrm{m}, T}}) \to \Pic_{X'/Y, \et} \to \Pic_{X/Y, \et} \to 0
\end{align*}
on the big étale site of $Y.$ \label{exactsequenceProp}
\end{proposition}
\begin{proof}
We proceed in a way similar to that of the proof of \cite{Ov}, Lemma 6.0.1. Let $\psi \colon X\to X'$ be the canonical finite morphism. Everything that follows must be shown after arbitrary base change $S\to Y.$ By Proposition \ref{Basechangeprop}, we have
$$X'_S=X_S\cup_{Z_S}T_S,$$ so in order to simplify the notation, we shall assume that $S=Y.$ Moreover, we have a closed immersion $\iota \colon T\to X'.$ Because pushforward commutes with flat base change (and hence, \it a fortiori, \rm with étale base change), we know that the morphism $\Og_{X'}\to \psi_{\ast} \Og_{X}$ is injective on the small étale site of $X'$ (by the last part of Proposition \ref{Basechangeprop}). Hence the same is true for the morphism ${\mathbf{G}_{\mathrm{m}, X'}}\to \psi_{\ast}{\mathbf{G}_{\mathrm{m}, X}}.$ Throughout this proof, we shall freely use that the sheaf $\xi_\ast \mathbf{G}_{\mathrm{m}, Z}$ is represented by $\Res_{Z/T} \mathbf{G}_{\mathrm{m}, Z},$ where $\xi$ denotes the morphism $Z\to T.$  In particular, we have a canonical morphism $\psi_{\ast}{\mathbf{G}_{\mathrm{m}, X}}\to \iota_\ast \xi_\ast{\mathbf{G}_{\mathrm{m}, Z}} = \iota_{\ast} \Res_{Z/T}{\mathbf{G}_{\mathrm{m}, Z}}.$ First, we claim that the kernel of the composition 
$$\psi_{\ast}{\mathbf{G}_{\mathrm{m}, X}}\to \iota_{\ast} \Res_{Z/T}{\mathbf{G}_{\mathrm{m}, Z}} \to  \iota_{\ast} ((\Res_{Z/T}{\mathbf{G}_{\mathrm{m}, Z}})/{\mathbf{G}_{\mathrm{m}, T}})$$ is equal to ${\mathbf{G}_{\mathrm{m}, X'}}$ on the small étale site of $X'.$ Let $U$ be étale over $X'$ and let $$\phi \in \psi_{\ast}{\mathbf{G}_{\mathrm{m}, X}}(U)={\mathbf{G}_{\mathrm{m}, X}}(X\times_{X'}U)$$ be a function which has trivial image in $\iota_{\ast} ((\Res_{Z/T}{\mathbf{G}_{\mathrm{m}, T}})/{\mathbf{G}_{\mathrm{m}, T}})(U).$ This means that the restriction of $\phi$ to $Z\times_{X'}U$ comes from $ T\times_{X'}U.$ Now Proposition \ref{BasechangepropII} tells us that $\phi$ pulls back from an invertible function on $U.$ This shows the inclusion "$\subseteq$"; the other inclusion is obvious.\\
Next, we claim that the morphism $\psi_{\ast}{\mathbf{G}_{\mathrm{m}, X}}\to \iota_{\ast} \Res_{Z/T}{\mathbf{G}_{\mathrm{m}, Z}}$ is surjective. Let $\widetilde{\iota}\colon Z \to X$ be the closed immersion. Clearly, the morphism ${\mathbf{G}_{\mathrm{m}, X}}\to \widetilde{\iota}_{\ast}{\mathbf{G}_{\mathrm{m}, Z}} $ is surjective. In particular, so is the morphism 
$$\psi_{\ast}{\mathbf{G}_{\mathrm{m}, X}} \to \psi_{\ast}\widetilde{\iota}_{\ast}{\mathbf{G}_{\mathrm{m}, Z}}=\iota_{\ast} \Res_{Z/T}{\mathbf{G}_{\mathrm{m}, Z}};$$ this follows from the fact that $\psi_{\ast}-$ is exact as $\psi$ is finite (\cite{Stacks}, Tag 04C2(4)). \\
Because closed immersions are finite, the same argument shows that the map $$\iota_{\ast}\Res_{Z/T}{\mathbf{G}_{\mathrm{m}, Z}}\to \iota_{\ast}((\Res_{Z/T}{\mathbf{G}_{\mathrm{m}, Z}})/{\mathbf{G}_{\mathrm{m}, T}})$$ is surjective. Hence we have shown that the sequence
$$0\to {\mathbf{G}_{\mathrm{m}, X'}} \to \psi_{\ast}{\mathbf{G}_{\mathrm{m}, X}} \to \iota_{\ast}((\Res_{Z/T}{\mathbf{G}_{\mathrm{m}, Z}})/{\mathbf{G}_{\mathrm{m}, T}})\to 0$$ is exact on the small étale site of $X'.$ This sequence induces the long exact sequence
\begin{align*}
0&\to f'_{\ast}{\mathbf{G}_{\mathrm{m}, X'}} \to f'_{\ast}\psi_{\ast}{\mathbf{G}_{\mathrm{m}, X}} \to f_{\ast}\iota_{\ast}((\Res_{Z/T}{\mathbf{G}_{\mathrm{m}, Z}})/{\mathbf{G}_{\mathrm{m}, T}})\\
 &\to \R^1f'_{\ast}{\mathbf{G}_{\mathrm{m}, X'}} \to \R^1f'_{\ast}\psi_{\ast}{\mathbf{G}_{\mathrm{m}, X}} \to \R^1f_{\ast}\iota_{\ast}((\Res_{Z/T}{\mathbf{G}_{\mathrm{m}, Z}})/{\mathbf{G}_{\mathrm{m}, T}})
\end{align*}
on the small étale site of $S.$ Clearly, $\R^1f'_{\ast}{\mathbf{G}_{\mathrm{m}, X'}}$ is the restriction of $\Pic_{X'/Y, \et}$ to the small étale site of $S.$  Since $\psi$ is finite, $\psi_{\ast}-$ is exact and we obtain $\R^1f'_{\ast}\psi_{\ast}{\mathbf{G}_{\mathrm{m}, X}}=\R^1f_{\ast}{\mathbf{G}_{\mathrm{m}, X}},$ which is the restriction of $\Pic_{X/Y, \et}$ to the small étale site of $S.$ Since $\iota$ is a closed immersion and $t$ is finite, we have 
$$R^1f_{\ast}\iota_{\ast}((\Res_{Z/T}{\mathbf{G}_{\mathrm{m}, Z}})/{\mathbf{G}_{\mathrm{m}, T}}) = \R^1 t_{\ast} ((\Res_{Z/T}{\mathbf{G}_{\mathrm{m}, Z}})/{\mathbf{G}_{\mathrm{m}, T}}) =0.$$ Hence the claim from the Proposition follows. 
\end{proof}\\
$\mathbf{Remark.}$ A similar exact sequence was obtained by Brion \cite{Brion}, Corollary 2.3, who used it for a different (but related) purpose. Brion's article pre-dates \cite{Ov}, but the conditions under which the result is obtained in \cite{Brion} are not quite right for the purposes of \cite{Ov} or the present article. This is why we have chosen to generalise \cite{Ov} rather than \cite{Brion}, where the language of \cite{Fer} is used. Moreover, our method of obtaining the exact sequence is different from Brion's and more direct, as we avoid using Raynaud's theory of rigidificators.\\
\\
Finally, let us give a condition under which cohomological flatness is preserved by the push-out construction. The condition will be far from optimal, but sufficient for our purposes.
\begin{lemma}
Let $f\colon X\to Y$ be a proper and flat morphism of schemes. Assume that $Y$ is a \rm Dedekind scheme, \it let $t\colon T\to Y$ and $z\colon Z\to Y$ be morphisms, and let $\widetilde{\iota}\colon Z\to X$ be a closed immersion. Finally, let $\xi\colon Z\to T$ be a morphism of schemes and assume that the standing assumption as well as assumptions (i),...,(v) from Paragraph \ref{pushoputpara} are satisfied, and that $f$ is cohomologically flat in dimension zero. Moreover, let $f'\colon X':=X\cup_ZT\to Y$ be the push-out of $\widetilde{\iota}\colon Z\to X$ along $\xi\colon Z\to T.$ Suppose that $f'_\ast \Og_{X'}=\Og_Y$ and that $f_\ast \Og_X$ is étale over $\Og_Y.$ \label{cohomflatlem}
Then $f'$ is cohomologically flat in dimension zero. 
\end{lemma}
\begin{proof}
By \cite{Liu}, Chapter 5.3, Corollary 3.22 together with \cite{Liu}, Chapter 5.3, Exercise 3.14, it is sufficient to show that the map $\Og_Y\to f'_\ast \Og_{X'}$ remains an isomorphism after the base change $i\colon \Spec \kappa(y) \to Y$ for all closed points $y\in Y.$ Choose such a $y.$ Let $X_y:=X\times_Y\Spec \kappa (y),$ and define $X'_y$ analogously. Then $X'_y$ is geometrically connected (by Stein factorisation), and we have 
$$\Gamma(X'_y, \Og_{X'_y})\subseteq \Gamma(X_y, \Og_{X_y})=i^\ast f_\ast \Og_X.$$ The inclusion follows from the last part of Proposition \ref{Basechangeprop}, and the equality is due to $f$ being cohomologically flat in dimension zero. Note that $i^\ast f_\ast \Og_X$ is geometrically reduced over $\kappa(y).$ In particular, $\Gamma(X'_y, \Og_{X'_y})$ is a geometrically connected étale $\kappa(y)$-algebra, which implies that the map $\kappa(y)\to \Gamma(X'_y, \Og_{X'_y})$ is an isomorphism, as claimed.
\end{proof}
\subsection{The structure of Jacobians over arbitrary fields}
Let $\kappa$ be an arbitrary field and let $C$ be a proper geometrically reduced curve over $\kappa.$ Moreover, let $\widetilde{C}$ and $C^{\mathrm{sn}}$ be the normalisation of $C$ and the seminormalisation of $C,$ respectively. As before, we use the notation $\nu\colon \widetilde{C}\to C,$ $\varsigma\colon C^{\mathrm{sn}}\to C,$ and $\widetilde{\varsigma}\colon \widetilde{C}\to C^{\mathrm{sn}}.$ We have $\nu=\varsigma\circ\widetilde{\varsigma}.$ If $G$ is a smooth connected commutative group scheme over $\kappa,$ we let $\mathrm{uni}(G)$ denote the maximal unirational subgroup of $G$ over $\kappa$ (cf. \cite{BLR}, p. 310), and we let $\mathscr{R}_{us, \kappa}(G)$ denote the maximal smooth connected split unipotent closed subgroup of $G$ (cf. \cite{CGP}, p. 63). Moreover, for a proper curve $D$ over $\kappa,$ we denote by $\Pic^0_{D/\kappa}$ the identity component of $\Pic_{D/\kappa},$ which we shall also call the \it Jacobian \rm of $D$ over $\kappa.$ Observe that $\Pic_{D/\kappa, \et}\cong \Pic_{D/\kappa, \fppf}$ is representable by Theorem \ref{representthm}, and that it is smooth over $\kappa$ by Proposition \ref{picsmoothprop}.
\begin{proposition}
We have $\mathrm{uni}(\Pic^0_{\widetilde{C}/\kappa})=0.$ \label{unitrivialprop}
\end{proposition}
\begin{proof}
By \cite{BLR}, Chapter 10.3, Theorem 1, we must show that any morphism of schemes $\varphi\colon U\to \Pic^0_{\widetilde{C}/\kappa}$ is constant, where $U$ is a non-empty open subset of $\mathbf{P}^1_\kappa.$ After replacing $\kappa$ by a finite separable extension, we may assume that $\widetilde{C}$ is geometrically integral and has a $\kappa$-point. Note that, by \cite{BLR}, Chapter 10.3, Remark 4, this does not affect the claim from the Proposition. Then $\varphi$ is induced by a line bundle $\mathscr{L}$ on $U\times_\kappa\widetilde{C}.$ Because $\Ps^1_\kappa$ is smooth over $\kappa,$ the scheme $\Ps^1_\kappa\times_\kappa\widetilde{C}$ is regular. In particular, $\mathscr{L}$ extends to a line bundle on $\Ps^1_\kappa\times_\kappa\widetilde{C},$ which means that $\varphi$ comes from a morphism $\overline{\varphi}\colon \Ps^1_{\kappa}\to \Pic^0_{\widetilde{C}/\kappa}.$ However, it is well-known that any morphism $\overline{\varphi} \colon \Ps^1_\kappa\to G$ is constant if $G$ denotes a smooth group scheme over $\kappa.$ Indeed, by Lüroth's theorem, we may otherwise replace $\overline{\varphi}$ by the normalisation of the scheme-theoretic image of $\Ps^1_\kappa$ in $G$ and assume that $\overline{\varphi}$ is an immersion generically. Then the morphism $\overline{\varphi}^\ast \Omega^1_{G/\kappa}\to \Omega^1_{\Ps^1_\kappa/\kappa}$ is generically surjective. However, $\Omega^1_{G/\kappa}$ is a free coherent sheaf on $G$ (\cite{BLR}, Chapter 4.2, Corollary 3), and since $\Omega^1_{\Ps^1_\kappa/\kappa}\cong \Og_{\Ps^1_\kappa}(-2),$ the morphism $\overline{\varphi}^\ast \Omega^1_{G/\kappa}\to \Omega^1_{\Ps^1_\kappa/\kappa}$ must vanish, which is absurd.
\end{proof}

\begin{theorem} With the notation from the beginning of this paragraph, let $\nu^\ast \colon \Pic^0_{C/\kappa}\to \Pic^0_{\widetilde{C}/\kappa},$ $\widetilde{\varsigma}^\ast \colon \Pic^0_{C^{\mathrm{sn}}/\kappa}\to \Pic^0_{\widetilde{C}/\kappa},$ and  $\varsigma^\ast \colon \Pic^0_{C/\kappa}\to \Pic^0_{C^{\mathrm{sn}}/\kappa}$ be the induced morphisms on Jacobians, which fit into the diagram
$$\begin{tikzcd}
\Pic^0_{C/\kappa} \arrow{rr}{\nu^\ast} \arrow[swap]{rd}{\varsigma^\ast}  && \Pic^0_{\widetilde{C}/\kappa}\\
& \Pic^0_{C^{\mathrm{sn}}/\kappa} \arrow[swap]{ru}{\widetilde{\varsigma}^\ast}.
\end{tikzcd}$$
Then all these morphisms are surjective in the étale topology, and we obtain a filtration 
$$0\subseteq \ker \varsigma^\ast \subseteq \ker \nu^\ast \subseteq \Pic^0_{C/\kappa}$$
by smooth connected closed subgroups. \label{Jacstructthm}
Moreover, we have $$\ker \varsigma^\ast=\mathscr{R}_{us, \kappa}(\Pic^0_{C/\kappa})$$ and $$\ker \nu^\ast =\mathrm{uni}(\Pic^0_{C/\kappa}).$$
\end{theorem}
\begin{proof}
We shall assume that $C$ is connected, which causes no loss of generality. The proof will be divided into several steps:\\
\\
\bf Sublemma A. \normalfont \it (a) Let $f\colon D\to C$ be a finite birational morphism of proper geometrically reduced curves over $\kappa.$ Suppose moreover that $f$ is a universal homeomorphism. Then the map $f^\ast \colon \Gamma(C, \Og_C)\to \Gamma(D, \Og_D)$ is an isomorphism. \\
(b) Let $f$ be as in (a) and suppose that $f$ induces isomorphisms of residue fields at all points of $D.$ Then the kernel of the homomorphism 
$$\Pic^0_{C/\kappa}\to \Pic^0_{D/\kappa}$$ is split unipotent and of dimension $\dim_{\kappa} \Gamma(C, f_\ast \Og_D/\Og_C).$ \\ \rm 
\\
\begin{proof} (a) The morphism $f^\ast$ is injective because $f$ is scheme-theoretically dominant. The $\kappa$-dimension of $\Gamma(D, \Og_D)$ is equal to the number of connected components of $D_{\kappa\alg};$ a similar formula holds for $C.$ However, since $f_{\kappa\alg}$ is an homeomorphism, this invariant coincides for $C$ and $D.$ \\
(b) Suppose we have a factorisation $D\to D' \to C$ of $D\to C$ into two finite birational universal homeomorphisms which induce isomorphisms on residue fields. Then we have an exact sequence
$$0 \to \ker(\Pic^0_{C/\kappa} \to \Pic^0_{D'/\kappa}) \to \ker(\Pic^0_{C/\kappa}\to \Pic^0_{D/\kappa}) \to \ker(\Pic^0_{D'/\kappa}\to \Pic^0_{D/\kappa})\to 0.$$
In particular, using Theorem \ref{FactorisationTheorem}(i), we may suppose that there is a closed point $x$ of $C$ and a closed immersion $\Spec \kappa(x)[\epsilon]/\langle \epsilon^2 \rangle \to D$ such that the diagram 
$$\begin{CD}
D@>>> C\\
@AAA@AAA\\
\Spec \kappa(x)[\epsilon]/\langle \epsilon^2 \rangle @>>> \Spec \kappa
\end{CD}$$ 
is co-Cartesian. Now Proposition \ref{exactsequenceProp} tells us that we have an exact sequence
\begin{align*} 0&\to \Res_{\kappa(x)/\kappa}((\Res_{(\kappa(x)[\epsilon]/\langle \epsilon^2 \rangle)/\kappa(x)} {\mathbf{G}_{\mathrm{m}, \kappa(x)[\epsilon]/\langle \epsilon^2 \rangle}})/{\mathbf{G}_{\mathrm{m}, \kappa(x)}}) \\& \to \Pic_{D/\kappa}\to \Pic_{C/\kappa}\to 0 \end{align*} in the étale topology. Because $$(\Res_{(\kappa(x)[\epsilon]/\langle \epsilon^2 \rangle)/\kappa} {\mathbf{G}_{\mathrm{m}, \kappa(x)[\epsilon]/\langle \epsilon^2 \rangle}})/{\mathbf{G}_{\mathrm{m}, \kappa(x)}} \cong \mathbf{G}_{\mathrm{a}, \kappa(x)}$$ and $\Res_{\kappa(x)/\kappa}\mathbf{G}_{\mathrm{a}, \kappa(x)} \cong \mathbf{G}_{\mathrm{a}, \kappa}^{[\kappa(x):\kappa]},$\footnote{This follows because $\kappa(x)$ is a free $\kappa$-module of rank $[\kappa(x):\kappa].$} we obtain an exact sequence
$$0\to \mathbf{G}_{\mathrm{a}, \kappa}^{[\kappa(x):\kappa]} \to \Pic_{D/\kappa}\to \Pic_{C/\kappa}\to 0,$$ again in the étale topology. This shows the first part of (b); the second part follows by considering the exact sequence $0\to \Og_C\to f_\ast \Og_D \to f_\ast \Og_D/\Og_C\to 0$ and using (a). 
\end{proof} \\
\bf Sublemma B. \normalfont \it Let $C$ be a seminormal proper geometrically reduced curve over $\kappa,$ and let $f\colon D\to C$ be a finite birational morphism. Then the kernel of $\Pic^0_{C/\kappa}\to \Pic^0_{D/\kappa}$ is unirational.  \rm \\
\\
\begin{proof}
Let $\Psi$ be the scheme-theoretic pre-image of $C^{\mathrm{sing}}$ in $D.$ We know from Theorem \ref{FactorisationTheorem} that $\Psi$ is reduced and that the map $D\to C$ is the push-out along $\Psi\to C^{\mathrm{sing}}.$ Hence, by Proposition \ref{exactsequenceProp}, we have an exact sequence
\begin{align*}
0 &\to \Res_{\Gamma(C, \Og_C)/\kappa}\mathbf{G}_{\mathrm{m}, \Gamma(C, \Og_C)} \to \Res_{\Gamma(D, \Og_D)/\kappa}\mathbf{G}_{\mathrm{m}, \Gamma(D, \Og_D)}\\
&\to \Res_{\Psi/\kappa} \mathbf{G}_{\mathrm{m}, \Psi}/\Res_{C^{\mathrm{sing}}/\kappa}\mathbf{G}_{\mathrm{m}, C^{\mathrm{sing}}} \to \Pic^0_{C/\kappa} \to \Pic^0_{D/\kappa} \to 0.
\end{align*}
By \cite{BLR}, Chapter 7.6, Proposition 2(i), the functor $\Res_{\Psi/\kappa} -$ preserves open immersions, so the scheme $\Res_{\Psi/\kappa}{\mathbf{G}_{\mathrm{m}, \Psi}}\subseteq \Res_{\Psi/\kappa}{\mathbf{A}^1_{\Psi}}\cong \mathbf{A}^{\dim_{\kappa} \Gamma(\Psi, \Og_{\Psi})}_{\kappa}$ is rational. This shows that $\ker \widetilde{\varsigma}^\ast$ is unirational.
\end{proof}\\
We can now return to the proof of Theorem \ref{Jacstructthm}. We have an exact sequence
$$0 \to\ker \varsigma^\ast \to \ker\nu^\ast \to \ker \widetilde{\varsigma}^\ast \to 0,$$ so Sublemmata A and B together with Lemma \ref{unirationallemma} show that $\ker \nu^\ast$ is indeed unirational. In particular, we have established the inclusions $\ker\varsigma^\ast \subseteq \mathscr{R}_{us, \kappa} (\Pic^0_{C/\kappa})$ and $\ker \nu^\ast \subseteq \mathrm{uni}(\Pic^0_{C/\kappa}).$\\
Since the quotient $\Pic^0_{\widetilde{C}/\kappa}=\Pic^0_{C/\kappa}/\ker \nu^\ast$ contains no unirational subgroups by Proposition \ref{unitrivialprop}, we deduce that $\ker \nu^\ast=\mathrm{uni}(\Pic^0_{C/\kappa}),$ as claimed. Now all that remains to be shown is that $\mathscr{R}_{us, \kappa}(\Pic^0_{C, \kappa})\subseteq\ker\varsigma^\ast.$  We know that $\mathscr{R}_{us, \kappa}(\Pic^0_{C, \kappa})\subseteq \mathrm{uni}(\Pic^0_{C/\kappa})=\ker \nu^\ast$ because $\mathscr{R}_{us, \kappa}(\Pic^0_{C, \kappa})$ is unirational. We must show that the image of $\mathscr{R}_{us, \kappa}(\Pic^0_{C/\kappa})$ in $\ker \widetilde{\varsigma}^\ast$ is trivial, for which it suffices to show that any homomorphism $\mathbf{G}_{\mathrm{a}, \kappa}\to \ker \widetilde{\varsigma}$ vanishes. Because $\Gamma(D, \Og_D)$ is an étale $\kappa$-algebra, the exact sequence from the proof of Sublemma B shows that $\ker \widetilde{\varsigma}^\ast$ is a quotient of $(\Res_{\Psi/\kappa}{\mathbf{G}_{\mathrm{m}, \Psi}})/(\Res_{C^{\mathrm{sn}, \mathrm{sing}}/\kappa}{\mathbf{G}_{\mathrm{m}, C^{\mathrm{sn}, \mathrm{sing}}}})$ by a torus.
Hence, by \cite{SGA3}, Exposé XVII, Théorème 6.1.1 A) ii), any homomorphism $\mathbf{G}_{\mathrm{a}, \kappa}\to \ker \widetilde{\varsigma}^\ast$ lifts to an homomorphism $\mathbf{G}_{\mathrm{a}, \kappa} \to \Res_{C^{\mathrm{sn}, \mathrm{sing}}/\kappa}((\Res_{\Psi/C^{\mathrm{sn}, \mathrm{sing}}}{\mathbf{G}_{\mathrm{m}, \Psi}})/{\mathbf{G}_{\mathrm{m}, C^{\mathrm{sn}, \mathrm{sing}}}}),$ which is the same as an homomorphism $\mathbf{G}_{\mathrm{a}, C^{\mathrm{sn}, \mathrm{sing}}}\to (\Res_{\Psi/C^{\mathrm{sn}, \mathrm{sing}}}{\mathbf{G}_{\mathrm{m}, \Psi}})/{\mathbf{G}_{\mathrm{m}, C^{\mathrm{sn}, \mathrm{sing}}}}$ over $C^{\mathrm{sn}, \mathrm{sing}}.$ Again by \cite{SGA3}, Exposé XVII, Théorème 6.1.1 A) ii), such a map lifts to an homomorphism $\mathbf{G}_{\mathrm{a}, \Psi}\to {\mathbf{G}_{\mathrm{m}, \Psi}},$ which must vanish because $\Psi$ is a reduced Artinian scheme. Hence our claim follows. 
\end{proof}
\begin{corollary}
Let $C$ be a proper, geometrically reduced curve over a field $\kappa.$ Then $C$ is seminormal if and only if $\mathscr{R}_{us, \kappa}(\Pic^0_{C/\kappa})=0.$ Moreover, $\mathrm{uni}(\Pic^0_{C/\kappa})=0$ if and only if the morphism 
$$\Pic^0_{C/\kappa}\to \Pic^0_{\widetilde{C}/\kappa}$$ is an isomorphism. \label{Jacstructcor}
\end{corollary}
\begin{proof}
The formula for the dimension of $\mathscr{R}_{us, \kappa}(\Pic^0_{C/K})$ (Sublemma B applied to the morphism $\varsigma$) implies the first part of the Corollary. The second part immediately follows from the equality $\ker \nu^\ast = \mathrm{uni}(\Pic^0_{C/\kappa})$ from Theorem \ref{Jacstructthm}. 
\end{proof}
\section{Néron models of Jacobians}
In this section, we shall construct Néron models of Jacobians of geometrically reduced curves. Throughout this section, $S$ will denote an \it excellent \rm Dedekind scheme and $K$ will denote the field of rational functions on $S.$ Since the results we shall prove are known if $\mathrm{char}\, K=0,$ we may assume that $p:=\mathrm{char}\, K>0.$ Because, in this case, $K$ is never perfect, we shall need the full force of the results established so far. Let $C$ denote a geometrically reduced proper curve over $K.$ Let $\widetilde{C}$ and $C^{\mathrm{sn}}$ be the normalisation and the seminormalisation of $C,$ respectively. Since we are in dimension 1, we know that $\widetilde{C}$ is regular. 
Let us begin by recalling the following result, due to J. Lipman:
\begin{theorem}
Let $\widetilde{f}_\eta\colon \widetilde{D}\to \Spec K$ be a proper regular curve over $K.$ Then there exists a proper, flat, and regular model $\widetilde{f}\colon \widetilde{\mathscr{D}} \to \Spec S$ of $\widetilde {D}$ (i. e., the morphism $\widetilde{f}$ is proper and flat, $\widetilde{\mathscr{D}}$ is regular, and the generic fibre of $\widetilde{f}$ is isomorphic to $\widetilde{D}$). \label{regmodelexistencetheorem}
\end{theorem}
\begin{proof}
We may assume that $\widetilde{D}$ is integral. By \cite{Stacks}, Tag 0A26, the morphism $\widetilde{f}_\eta$ is projective. By taking the Zariski closure in a suitable projective space over $S,$ we can find a projective model of $\widetilde{D},$ which will be excellent because it is of finite type over $S$ (\cite{Stacks}, Tag 07QU). Moreover, the model will be two-dimensional and integral. Now Lipman's theorem on desingularization of surfaces (\cite{Lip}, Theorem on p. 151) guarantees the existence of our desired model (observe that the generic fibre remains unaffected by the desingularization morphism, since that morphism is proper and birational).
\end{proof}\\
We shall need some auxiliary results on maximal separated quotients of group algebraic spaces over $S.$ Let $G\to S$ be a smooth commutative group object in the category of (locally separated and quasi-separated) algebraic spaces over $S.$ Following \cite{Ray}, Proposition 3.3.5, we consider the scheme-theoretic closure $E\subseteq G$ of the unit section $S\to G$ of $G.$ Then we have the following
\begin{lemma}
Assume that there exists an open dense subset of $U\subseteq S$ above which $G$ is separated. Then the quotient $G^{\mathrm{sep}}:=G/E$ is a scheme, which is smooth and separated over $S.$ Moreover, $E$ is étale over $S.$ \label{EetalelemII}
\end{lemma}
\begin{proof}
Observe first that $E\to G$ is a closed immersion by construction. In particular, the quotient $G/E$ exists as an algebraic space over $S$ by \cite{BLR}, Chapter 8.3, Proposition 9. Let $s$ be a closed point of $S$ not contained in $U.$ Since scheme-theoretic images of quasi-compact morphisms of algebraic spaces commute with flat base change (\cite{Stacks}, Tag 089E), and since fppf-quotients commute with arbitrary base change, we deduce from \cite{Ray}, Proposition 3.3.5 that $G/E\times_S\Spec\Og_{S,s}$ is separated. Because separatedness is local on the base in the fpqc-topology (\cite{Stacks}, Tag 0421), it follows that $G/E$ is a separated group object in the category of algebraic spaces over $S,$ and hence a scheme by \cite{An}, Théorème 4.B. A completely analogous argument shows that $E\times_S\Spec \Og_{S,s}$ is étale over $\Spec \Og_{S,s}$ for all $s$ as above (\cite{Ray}, Proposition 3.3.5), and since being étale is local on the base in the fpqc-topology as well (\cite{Stacks}, Tag 042B), we find that $E$ is étale over $S.$ 
\end{proof}\\
$\mathbf{Remark}.$ The lemma above would fail completely if $S$ were of dimension greater than 1, since $E$ would not necessarily be flat over $S$ in this case (see \cite{Hol} for more details).
\begin{proposition}
Let $\widetilde{C}$ be a geometrically reduced regular proper curve over $K.$ \\
(i) Suppose $\widetilde{C}$ is integral and $\widetilde{C}(K)\neq\emptyset.$ Then, for any proper, flat, and regular model $\widetilde{\mathscr{C}}\to S$ of $\widetilde{C},$  $P_{\widetilde{\mathscr{C}}/S}^{\mathrm{sep}}$ is\footnote{See Lemma \ref{EetalelemII} for the notation $(-)\sep.$} a Néron model of $\Pic^0_{\widetilde{C}/K}.$\\
(ii) In general, $\Pic^0_{\widetilde{C}/K}$ admits a Néron model over $S.$ \label{regularnéronprop}
\end{proposition}
\begin{proof}
We begin by proving (i). By Theorem \ref{regmodelexistencetheorem}, we may choose a proper, flat, and regular model $\widetilde{\mathscr{C}}\to S$ of $\widetilde{C}.$ As a first step, we show that $\widetilde{\mathscr{C}}\to S$ is cohomologically flat in dimension zero. By \cite{Liu}, Chapter 5.3, Exercise 3.14 (a), we may assume that $S$ is the spectrum of a discrete valuation ring. Because $\widetilde{C}$ has a $K$-rational point and $\widetilde{\mathscr{C}}$ is proper over $S,$ the map $\widetilde{\mathscr{C}}\to S$ admits a section. Since $\widetilde{\mathscr{C}}$ is regular, the section factors through the smooth locus of the map $\widetilde{\mathscr{C}}\to S.$ Hence the special fibre of $\widetilde{\mathscr{C}}\to S$ has an irreducible component of geometric multiplicity one, so the claim follows from \cite{Liu}, Chapter 9.1, Corollary 1.24 and Remark 1.25, as well as \cite{Liu}, Chapter 8.3, Theorem 3.16. In particular, $\Pic_{\widetilde{\mathscr{C}}/S}$ is representable by a smooth algebraic space over $S$ by Theorem \ref{representthm} and Proposition \ref{picsmoothprop}.
We let $P_{\widetilde{\mathscr{C}}/S}$ denote the kernel of the degree map
$$\deg \colon \Pic_{\widetilde{\mathscr{C}}/S}\to \Z.$$ By \cite{BLR}, Chapter 9.2, Corollary 14, the generic fibre of $P_{\widetilde{\mathscr{C}}/S}$ is equal to $\Pic^0_{\widetilde{C}/K}.$ We claim that $P_{\widetilde{\mathscr{C}}/S}^{\mathrm{sep}}$ is the Néron model of $\Pic^0_{\widetilde{C}/K}.$ \\
We shall prove first that  $P_{\widetilde{\mathscr{C}}/S}^{\mathrm{sep}}$ is the Néron lft-model of $\Pic^0_{\widetilde{C}/K},$ and then show in a second step that it is of finite type over $S.$ The first claim follows from Lemma \ref{localgloballem} together with \cite{BLR}, Chapter 9.5, Theorem 4. By \cite{BLR}, Chapter 10.1, Corollary 10, we may conclude the proof by showing that, for all but finitely many $s\in S,$ the Néron model of $\Pic^0_{\widetilde{C}/K}$ over $\Spec \Og_{S,s}$ has connected special fibre, and that the groups of connected components at the remaining fibres are finite. The second claim is a consequence of \cite{BLR}, Chapter 9.5, Theorem 4. Because $\widetilde{C}$ is geometrically integral, there exists an open dense subset $U\subseteq S$ such that the fibres of $\widetilde{\mathscr{C}}\times_SU\to U$ are geometrically integral. This follows from \cite{Stacks}, Tags 055G and 0578. For all $s\in U,$ the the special fibre of the Néron model of $\Pic^0_{\widetilde{C}/K}$ over $\Spec \Og_{S,s}$ is connected by \cite{BLR}, Chapter 9.5, Theorem 1 (note that $\widetilde{\mathscr{C}}\times_S\Spec \Og_{S,s}$ is projective over $\Og_{S,s}$ by \cite{Liu}, Chapter 8.3, Theorem 3.16). Hence part (i) of the Theorem follows. \\
Part (ii) follows from part (i). Indeed, after replacing $K$ by a finite separable extension if necessary, we may assume that every irreducible component of $\widetilde{C}$ has a $K$-point. Since $\Pic^0_{-/K}$ transforms disjoint unions into products, we may assume that the conditions of part (i) are satisfied and conclude using \cite{BLR}, Chapter 7.2, Proposition 4.
\end{proof}\\
We are now in a position to give a positive answer to Conjecture II from \cite{BLR}, Chapter 10.3, in the case of Jacobians of geometrically reduced proper curves:
\begin{theorem} \rm (Theorem \ref{ConjIIthmInt}) \it
Let $S$ be an excellent Dedekind scheme with field of rational functions $K.$ Let $C$ be a proper, geometrically reduced curve over $K.$ Let $G_K:=\Pic^0_{C/K}$ and suppose that $\mathrm{uni}(G_K)=0.$ Then $G_K$ admits a Néron model over $S.$ \label{ConjIIthm}
\end{theorem}
\begin{proof}
By Corollary \ref{Jacstructcor}, we may assume without loss of generality that $C=\widetilde{C}$ is regular. Hence the Theorem follows from Proposition \ref{regularnéronprop}.
\end{proof}\\
The proof of Conjecture I from \cite{BLR}, Chapter 10.3 for Jacobians of geometrically reduced curves is more difficult and will occupy the remainder of this Section. The idea we shall pursue will use a generalisation of the construction presented in Section 5 of \cite{Ov}. As in \it loc. cit., \rm we shall begin by constructing \it good models of singular curves. \rm We begin with the following 
\begin{lemma}
Let $S$ be a Dedekind scheme with field of rational functions $K.$ Let $C$ be a seminormal proper geometrically reduced curve over $K.$ Let $\Psi:=\widetilde{C}\times_CC^{\mathrm{sing}}.$ Let ${\overline{\Psi}}$ be the integral closure of $S$ in $\Psi.$ Then there exists a proper regular flat model $\widetilde{f}\colon \widetilde{\mathscr{C}}\to S$ of $\widetilde{C}$ together with a closed immersion \label{singimmlem}
$$\iota\colon {\overline{\Psi}}\to \widetilde{\mathscr{C}}.$$ 
\end{lemma}
\begin{proof}
Choose a proper, flat, and regular model $\widetilde{\mathscr{C}}'\to S$ of $\widetilde{C}$ as in Theorem \ref{regmodelexistencetheorem}. Note that ${\overline{\Psi}}$ is finite over $S$ because $\Psi$ is reduced (see Theorem \ref{FactorisationTheorem}(ii)) and $S$ is excellent. Because ${\overline{\Psi}}$ is a disjoint union of Dedekind schemes and because $\widetilde{f}$ is proper, we obtain a canonical morphism ${\overline{\Psi}}\to \widetilde{\mathscr{C}}'$ which extends $\Psi\to C.$ The scheme-theoretic image $D$ of ${\overline{\Psi}}\to \widetilde{\mathscr{C}}'$ is a reduced divisor on $\widetilde{C}.$ Moreover, $D$ is clearly excellent as a scheme. By the embedded resolution theorem (\cite{Liu}, Chapter 9.2, Theorem 2.26), we can find a proper birational morphism $\varphi\colon \widetilde{\mathscr{C}}\to \widetilde{\mathscr{C}}'$ such that $\varphi^\ast D$ has strict normal crossings. Then the strict transform of $\widetilde{D}$ of $D$ is regular (\cite{Liu}, Chapter 9.2, Remark 2.27), so the induced map ${\overline{\Psi}} \to \widetilde{D}$ is an isomorphism.
\end{proof}

\begin{lemma}
With the notation from the preceding lemma, Zariski locally on $S,$ the map $\iota$ factors through an open affine subscheme of $\widetilde{\mathscr{C}}.$ \label{Standinglem}
\end{lemma}
\begin{proof}
We may assume without loss of generality that the map $\widetilde{\mathscr{C}}\to S$ is projective (\cite{Liu}, Chapter 8.3, Theorem 3.16). Let $s\in S$ be a  point. Because ${\overline{\Psi}}$ is finite over $S,$ we can find an open affine subset $V$ of $\widetilde{\mathscr{C}}$ which contains the fibre of ${\overline{\Psi}}\to S$ above $s.$ Let $Z$ be the complement of $V\cap {\overline{\Psi}}$ in ${\overline{\Psi}}.$ Because the morphism ${\overline{\Psi}}\to S$ is finite, the image of $Z$ in $S$ is closed in $S.$ Let $U$ be the complement of that image. Now we replace $U$ by an open affine neighbourhood of $s$ in $U.$ Then $U$ has the desired property. Indeed, the morphism $V\to S$ is affine, so the preimage $V_U$ of $U$ in $V$ is affine, and it is easily verified that $\iota$ factors through $V_U$ above $U.$ 
\end{proof}

\begin{corollary}
Let $C$ be a proper connected geometrically reduced seminormal curve over $K.$ Suppose that every irreducible component of $C$ admits a $K$-rational point smooth on $C.$ Then there exists a flat proper model $f\colon \mathscr{C}\to S$ of $C$ which is cohomologically flat in dimension zero. \label{Cexistcor}
\end{corollary}
\begin{proof}
Let $\widetilde{C}$ denote the normalisation of $C$ and let $\widetilde{f}\colon \widetilde{\mathscr{C}}\to S$ be the proper regular model of $\widetilde{C}$ together with the closed immersion ${\overline{\Psi}}\to \widetilde{\mathscr{C}}$ from the Lemma \ref{singimmlem}. Moreover, let $\mathscr{C}^{\mathrm{sing}}$ be the integral closure of $S$ in $C^{\mathrm{sing}}.$ Observe that we have a canonical morphism ${\overline{\Psi}}\to \mathscr{C}^{\mathrm{sing}}.$
We would like to construct a model of $C$ by defining
$$\mathscr{C}:=\widetilde{\mathscr{C}}\cup_{{\overline{\Psi}}}\mathscr{C}^{\mathrm{sing}},$$ so we must check the standing assumption as well as conditions (i),...,(v) from Paragraph \ref{pushoputpara}. Conditions (i), (ii), and (iii) are clearly satisfied, and the standing assumption from that Paragraph follows from Lemma \ref{Standinglem}. In particular, $\mathscr{C}$ is proper over $S.$ Moreover $\mathscr{C}$ is flat over $S$ by Proposition \ref{Basechangeprop}(b). Now observe that the map ${\overline{\Psi}}\to \mathscr{C}^{\mathrm{sing}}$ is flat because $\mathscr{C}^{\mathrm{sing}}$ is a disjoint union of Dedekind schemes and all generic points of ${\overline{\Psi}}$ map to generic points of $\mathscr{C}^{\mathrm{sing}}.$ Since ${\overline{\Psi}}\to \mathscr{C}^{\mathrm{sing}}$ is finite and generically surjective, we deduce that this map is faithfully flat (showing Assumption (iv)). Now write $z$ and $t$ for the maps ${\overline{\Psi}}\to S$ and $\mathscr{C}^{\mathrm{sing}}\to S,$ respectively. To verify assumption (v), we must check that the quotient of $t_\ast \Og_{\mathscr{C}^{\mathrm{sing}}}\to z_\ast \Og_{{\overline{\Psi}}}$ is locally free over $S.$ This follows from the following\\
\\
$\mathbf{Sublemma}.$ \it Let $R$ be a Dedekind domain with field of fractions $K.$ Let $K\subseteq L \subseteq F$ be finite reduced $K$-algebras, and let $B\subseteq F$ be a finite $R$-submodule. Moreover, let $A:= B\cap L.$ Then $B/A$ is finite and locally free over $R.$ \rm \\
\\
The proof of the Sublemma will be left to the reader. We have now seen that the standing assumption as well as assumption (i),...,(v) from Paragraph \ref{pushoputpara} are satisfied. Because $C$ is geometrically reduced and geometrically connected (as it has a $K$-point), we find that $\Gamma(C, \Og_C)=K,$ which implies that $f_\ast \Og_{\mathscr{C}}=\Og_S$ because $f_\ast \Og_{\mathscr{C}}$ is finite and flat over $\Og_S.$  Moreover, since all irreducible components of $C$ have smooth (and hence regular) $K$-points, the same is true for $\widetilde{C}.$ In particular, we have
$\widetilde{f}_\ast \Og_{\widetilde{\mathscr{C}}}\cong \Og_S^n$ as an $\Og_S$-algebra for some $n\in \N,$ and this remains true after any base change (see the proof of Proposition \ref{regularnéronprop}). This implies that $\widetilde{f}_\ast \Og_{\widetilde{\mathscr{C}}}$ is étale over $\Og_S$, and Lemma \ref{cohomflatlem} tells us that $\mathscr{C}$ is cohomologically flat in dimension zero over $S,$ which concludes the proof. 
\end{proof}\\
$\mathbf{Remark}.$ The notation $\mathscr{C}^{\mathrm{sing}}$ makes sense because the closed immersion $\mathscr{C}^{\mathrm{sing}}\to \mathscr{C}$ identifies $\mathscr{C}^{\mathrm{sing}}$ with the singular locus of $\mathscr{C}.$

\begin{corollary}
Let $C$ be as in Corollary \ref{Cexistcor} and let $\mathscr{C}$ be the model of $C$ constructed there. Put $S':=\boldsymbol{\Spec}\,\widetilde{f}_\ast \Og_{\widetilde{\mathscr{C}}}.$ Then we have an exact sequence 
\begin{align}
0&\to (\Res_{S'/S}{\mathbf{G}_{\mathrm{m}, S'}})/{\mathbf{G}_{\mathrm{m}, S}} \overset{j}{\to} (\Res_{{\overline{\Psi}}/S}{\mathbf{G}_{\mathrm{m}, {\overline{\Psi}}}})/(\Res_{\mathscr{C}^{\mathrm{sing}}/S}{\mathbf{G}_{\mathrm{m}, \mathscr{C}^{\mathrm{sing}}}}) \nonumber\\ &\to \Pic_{\mathscr{C}/S}\overset{\pi}{\to} \Pic_{\widetilde{\mathscr{C}}/S}\to 0 \label{exactsequenceII}
\end{align}
of (not necessarily separated) group objects in the category of algebraic spaces over $S$ in the étale topology.
\end{corollary}
\begin{proof}
This is a consequence of Proposition \ref{exactsequenceProp}. Indeed, because $\mathscr{C}$ and $\widetilde{\mathscr{C}}$ are proper and flat, their étale and fppf-Picard functors are isomorphic (Proposition \ref{Picisomprop}), so both of them are representable by Theorem \ref{representthm}. We must also show that the standing assumption and assumptions (i),...,(v) from Paragraph \ref{pushoputpara} are satisfied, which we have already verified in the proof of Corollary \ref{Cexistcor}. Let $z\colon {\overline{\Psi}}\to S$ and $t\colon \mathscr{C}^{\mathrm{sing}}\to S$ denote the structure morphisms. All we must verify is that $t_\ast((\Res_{{\overline{\Psi}}/\mathscr{C}^{\mathrm{sing}}}{\mathbf{G}_{\mathrm{m}, {\overline{\Psi}}}})/{\mathbf{G}_{\mathrm{m}, \mathscr{C}^{\mathrm{sing}}}})$ is isomorphic to $(\Res_{{\overline{\Psi}}/S}{\mathbf{G}_{\mathrm{m}, {\overline{\Psi}}}})/(\Res_{\mathscr{C}^{\mathrm{sing}}/S}{\mathbf{G}_{\mathrm{m}, \mathscr{C}^{\mathrm{sing}}}}).$ This is a consequence of \cite{CGP}, Corollary A.5.4 (3) because $t$ is a finite flat morphism between Noetherian schemes.
\end{proof}\\
We shall now construct a Néron lft-model of $\Pic^0_{C/K},$ proceeding in a way similar to \cite{Ov}, Proposition 6.0.2 and Theorem 6.0.6. Let us write
$$\mathscr{K}_1:=\ker (\Pic_{\mathscr{C}/S} \overset{\pi}{\to} \Pic_{\widetilde{\mathscr{C}}}).$$ with the notation from the exact sequence (1). We begin with the following technical
\begin{lemma}
Keep the notation and assumptions from above. Then there exists an open dense subset of $S$ above which $\mathscr{K}_1,$ $\Pic_{\mathscr{C}/S},$ and $\Pic_{\widetilde{\mathscr{C}}/S}$ are all separated. Moreover, $\mathscr{K}_1$ is a smooth algebraic space over $S.$ \label{K1seperatedlem}
\end{lemma}
\begin{proof}
It suffices to exhibit such an open subset for each of the above algebraic spaces individually, since their intersection will then have the desired property. Moreover, once we have found such open subsets for $\mathscr{K}_1$ and $\Pic_{\widetilde{\mathscr{C}}/S},$ the same will do for $\Pic_{\mathscr{C}/S}.$ First observe that we have
$$\mathscr{K}_1=\mathrm{coker}\, j$$
with the notation from (1). Note that $\mathscr{K}_1$ is representable by an algebraic space. Because $j$ is generically a closed immersion (\cite{Stacks}, Tag 047T), the same is true for the restriction of $j$ to some dense open subset of $S$ (\cite{Stacks}, Tag 01ZP). As for $\Pic_{\widetilde{\mathscr{C}}/S},$ we reduce to the case where $\widetilde{\mathscr{C}}$ is integral, as it is the disjoint union of finitely many integral schemes. Hence we know that, above some dense open subset $V$ of $S,$ $\widetilde{\mathscr{C}}$ has geometrically integral fibres. Therefore $\Pic_{\widetilde{\mathscr{C}}/S}$ is separated above $V.$ Indeed, by \cite{Liu}, Chapter 8.3, Theorem 3.16, $\widetilde{\mathscr{C}}$ is projective over $V$ (after possibly shrinking $V$). Hence the claim follows from \cite{BLR}, Chapter 8.2, Theorem 1. Finally, note that $\mathscr{K}_1$ is a quotient of a smooth algebraic space and hence itself smooth. This proves our claim.
\end{proof}\\
Now let $C$ be a geometrically reduced proper curve over $K.$ Assume that every irreducible component of $C$ admits a smooth $K$-point. Let $C_1,..., C_n$ denote the irreducible components of $C,$ which are geometrically integral over $K.$ Let $\mathscr{C}$ be the model of $C$ constructed in Corollary \ref{Cexistcor}. By construction, the scheme $\mathscr{C}$ has reduced irreducible components $\mathscr{C}_j,$ such that, for all $j=1,..., n,$ $\mathscr{C}_j$ is a proper and flat model of $C_j$ over $S.$ 
We define a morphism 
$$\deg\colon \Pic_{\mathscr{C}/S}\to \underline{\Z}^n$$ as follows (where $\underline{\Z}$ denotes the sheafification of the constant presheaf $\Z$): for each morphism $T\to S,$ we have a map 
\begin{align*}
\Pic(\mathscr{C}\times_ST)\to \underline{\Z}^n(T)
\end{align*}
coming from the fact that for each $\mathscr{L}\in \Pic(\mathscr{C}\times_ST)$ and each $j=1,..., n,$ the map $T\to \Z,$ $t\mapsto \deg \mathscr{L}\mid_{\mathscr{C}_{j,t}}$ is locally constant (\cite{BLR}, Chapter 9.1, Proposition 2). We define a map
$$\widetilde{\deg}\colon \Pic_{\widetilde{\mathscr{C}}/S}\to \Z^n$$ entirely analogously. Because for all $j=1,..., n,$ the map $\widetilde{\mathscr{C}}_j\to \mathscr{C}_j$ is the normalisation morphism, it follows from the calculation in \cite{BLR}, p. 237f that we have 
$$\deg=\widetilde{\deg}\circ \pi$$ generically, so this equality follows everywhere because $\underline{\Z}^n$ is separated over $S.$ Define
$$P_{\mathscr{C}/S}:=\ker \deg$$
and 
$$P_{\widetilde{\mathscr{C}}/S}:=\ker \widetilde{\deg}.$$ We immediately obtain an exact sequence
$$0\to \mathscr{K}_1\to P_{\mathscr{C}/S}\to P_{\widetilde{\mathscr{C}}/S}\to 0$$ in the étale topology over $S.$ By \cite{BLR}, Chapter 9.3, Corollary 14, $P_{\mathscr{C}/S}$ and $P_{\widetilde{\mathscr{C}}/S}$ are models of $\Pic^0_{C/K}$ and $\Pic^0_{\widetilde{C}/K},$ respectively. We have now assembled all the technical tools needed to give a positive answer to Conjecture I of \cite{BLR}, Chapter 10.3, for Jacobians of geometrically reduced curves:
\begin{theorem} \rm (Theorem \ref{ConjIthmInt}) \it
Let $S$ be an excellent Dedekind scheme with field of rational functions $K.$ Let $C$ be a proper geometrically reduced curve over $K.$ Let $G_K:=\Pic^0_{C/K}$ and assume that $\mathscr{R}_{us, K}(G_K)=0$ (or, equivalently, that $C$ is seminormal; see Corollary \ref{Cexistcor}.). Then $G_K$ admits a Néron lft-model over $S.$ \label{ConjIthm}
\end{theorem}
\begin{proof}
The proof is based on \cite{Ov}, Proposition 6.0.4 and Theorem 6.0.6, and will be divided into several steps. Note first that we have an exact sequence
\begin{align}0 \to \mathrm{uni}_K(\Pic^0_{C/K} ) \to \Pic^0_{C/K} \to \Pic^0_{\widetilde{C}/K}\to 0\label{310seq}\end{align} by Theorem \ref{Jacstructthm}. \\
\it Step 1: \rm Without loss of generality, we may suppose that $C$ is connected, because $\Pic^0_{-/K}$ transforms disjoint unions into products. Moreover, we may assume that each irreducible component of $C$ has a $K$-rational point smooth on $C$. Indeed, because $C$ is geometrically reduced, it contains a smooth dense open subscheme. In particular, after replacing $K$ by a siutable finite \it separable \rm extension, each irreducible component of $C$ acquires a smooth rational point. We replace $S$ by its integral closure in the separable extension we chose, which is still an excellent Dedekind scheme. This does not lead to any loss of generality by \cite{BLR}, Chapter 10.1, Proposition 4. Moreover, we still have $\mathscr{R}_{us, K}(G_K)=0$ by \cite{CGP}, Corollary B.3.5, and by Corollary \ref{Jacstructcor}, $C$ remains seminormal after the base change. Moreover, sequence (\ref{310seq}) commutes with base change along finite separable extensions. \\
\it Step 2: \rm We shall show that that the sequence (\ref{310seq}) can be extended to an exact sequence 
\begin{align}0\to \mathscr{U} \to \mathscr{N} \to \widetilde{\mathscr{N}}\to 0\label{310seq2}\end{align} of group schemes over $S$ such that $\mathscr{U}$ and $\widetilde{\mathscr{N}}$ are the Néron lft-models of $\mathrm{uni}_K(\Pic^0_{C/K})$ and $\Pic^0_{\widetilde{C}/K},$ respectively. Once this is achieved, the proof is concluded by Corollary \ref{exactsequenceNéroncor}. \\
\it Step 3: \rm In order to construct sequence (\ref{310seq2}), it is sufficient to construct an exact sequence 
\begin{align}0 \to \mathscr{U}'\to \mathscr{N}'\to \widetilde{\mathscr{N}} \to 0 \label{310seq3}\end{align} of $S$-group schemes extending (\ref{310seq}) such that $\widetilde{\mathscr{N}}$ is the Néron lft-model of $\Pic^0_{\widetilde{C}/K}$ and such that $\mathscr{U}'$ is smooth over $S.$ To see this, note first that $\mathrm{uni}_K(\Pic^0_{C/K})$ admits a Néron lft-model $\mathscr{U}$ over $S.$ Indeed, with the notation from the proof of Theorem \ref{Jacstructthm}, we have seen that $\mathrm{uni}(\Pic^0_{C/K})$ is equal to the quotient of $(\Res_{\Psi/K}{\mathbf{G}_{\mathrm{m}, \Psi}})/(\Res_{C^{\mathrm{sing}}/K}{\mathbf{G}_{\mathrm{m}, C ^ \mathrm{sing}}})$ by the torus $(\Res_{\Gamma(\widetilde{C}, \Og_{\widetilde{C}})/K}{\mathbf{G}_{\mathrm{m}, \Gamma(\widetilde{C}, \Og_{\widetilde{C}})}})/{\mathbf{G}_{\mathrm{m}, K}}$ (recall that $C$ is seminormal, so $\nu=\widetilde{\varsigma}$). However, because each irreducible component of $\widetilde{C}$ has a $K$-point, the second torus is split. Therefore the existence of $\mathscr{U}$ follows from Lemma \ref{uninéronexistlem} and Proposition \ref{Gmclimprop}. By the Néron mapping property, there is a unique morphism $\mathscr{U}'\to \mathscr{U}$ extending the identity at the generic fibre. Consider the push-out diagram
$$\begin{CD}
0 @>>> \mathscr{U}' @>{\alpha}>> \mathscr{N}' @>>> \widetilde{\mathscr{N}} @>>>0 \\
&& @V{\beta}VV@VVV@VVV\\
0 @>>> \mathscr{U} @>>> \mathscr{U}\oplus_{\mathscr{U}'} \mathscr{N}' @>>> \widetilde{\mathscr{N}}@>>> 0
\end{CD}$$
in the category of fppf-sheaves over $S.$ Then $\mathscr{U}\oplus_{\mathscr{U}'} \mathscr{N}'$ is isomorphic to the cokernel of the morphism $\mathscr{U}' \to \mathscr{U} \oplus \mathscr{N}'$ given by $x\mapsto (-\beta(x), \alpha(x)),$ which is a closed immersion because so is $\alpha$ and $\mathscr{U} \oplus \mathscr{N}'$ is separated. Therefore, putting $\mathscr{N}:=\mathscr{U}\oplus_{\mathscr{U}'} \mathscr{N}',$ the sheaf $\mathscr{N}$ is representable by \cite{An}, Théorème 4.C and we obtain sequence (\ref{310seq2}). \\
\it Step 4: \rm We shall now construct sequence (\ref{310seq3}). Let $\widetilde{\mathscr{C}}\to S$ be the proper flat model of $\widetilde{C}$ from Theorem \ref{regmodelexistencetheorem}, and let $\mathscr{C}$ be the proper flat model of $C$ constructed in Corollary \ref{Cexistcor}. 
Let $\mathscr{K}_2$ be the kernel of the induced map 
$P_{\mathscr{C}/S}^{\mathrm{sep}}\to P_{\widetilde{\mathscr{C}}/S}^{\mathrm{sep}}.$ (See Lemma \ref{EetalelemII} for the notation used here.) By \cite{An}, Théorème 4.B, both $P_{\mathscr{C}/S}^{\mathrm{sep}}$ and $P_{\widetilde{\mathscr{C}}/S}^{\mathrm{sep}}$ are schemes; hence so is $\mathscr{K}_2.$ We obtain a commutative diagram
\begin{align}\begin{CD}
0@>>>\mathscr{K}_1@>>> P_{\mathscr{C}/S}@>>>P_{\widetilde{\mathscr{C}}/S}@>>>0\\
&&@VVV@VVV@VVV\\
0@>>>\mathscr{K}_2@>>>P_{\mathscr{C}/S}^{\mathrm{sep}}@>>>P_{\widetilde{\mathscr{C}}/S}^{\mathrm{sep}}@>>>0. \label{PsepCD}
\end{CD}\end{align}
with exact rows in the category of fppf-sheaves on $S.$ Now we claim that the morphism $P_{\mathscr{C}/S}^{\mathrm{sep}}\to P_{\widetilde{\mathscr{C}}/S}^{\mathrm{sep}}$ is smooth. Let $\mathscr{E}_2$ and $\mathscr{E}_3$ be the kernels of the second and third horizontal arrows. Lemmata  \ref{EetalelemII} and \ref{K1seperatedlem} show that $\mathscr{E}_2$ and $\mathscr{E}_3$ are étale over $S.$ By Lemma \ref{K1seperatedlem}, the map $P_{\mathscr{C}/S} \to P_{\widetilde{\mathscr{C}}/S}$ is smooth. Because smoothness is local on the source and target in the smooth topology (\cite{Stacks}, Tags 06F2 and 0429), and because the maps  $P_{\mathscr{C}/S} \to P_{\mathscr{C}/S}^{\mathrm{sep}}$ and $P_{\widetilde{\mathscr{C}}/S} \to P_{\widetilde{\mathscr{C}}/S}^{\mathrm{sep}}$ are surjective by construction, we obtain the desired smoothness of $P_{\mathscr{C}/S}^{\mathrm{sep}} \to P_{\widetilde{\mathscr{C}}/S}^{\mathrm{sep}}.$ Finally, we know from Lemma \ref{localgloballem} together with \cite{BLR}, Chapter 9.5, Theorem 4, that $P_{\widetilde{\mathscr{C}}/S}^{\mathrm{sep}}$ is the Néron model of $\Pic^0_{\widetilde{C}/K}.$ This shows that the bottom sequence in diagram (\ref{PsepCD}) is of type (\ref{310seq3}), concluding Step 4. 
\end{proof}\\
$\mathbf{Remark}.$ Let $C$ be a geometrically reduced proper seminormal curve over $K.$ If each irreducible component of $C$ admits a $K$-rational point regular on $C,$ then the proof above shows that the exact sequence $0\to \mathrm{uni}(\Pic^0_{C/K})\to \Pic^0_{C/K}\to \Pic^0_{\widetilde{C}/K}\to 0$ of $K$-group schemes induces an exact sequence of Néron lft-models. It is well-known that, in general, Néron lft-models behave very badly in exact sequences. It would be very interesting to understand this behaviour in general, and in particular whether the sequence is always exact. 
\section{Semi-factorial models of geometrically integral curves}
Let $S$ be a Dedekind scheme with field of rational functions $K$ and let $C$ be a proper, geometrically reduced curve over $K.$ In the proof of Theorem \ref{ConjIthm} (where $S$ is excellent), we constructed a Néron model of $\Pic^0_{C/K}$ by considering a proper and flat model $\mathscr{C}\to S$ which is cohomologically flat in dimension zero (at least after a finite extension of $K$), considering the scheme $P_{\mathscr{C}/S}^{\mathrm{sep}},$ and then employing a push-out construction. In the case where $C$ is regular, the last step is unnecessary. Hence we shall now investigate under which circumstances the $S$-scheme $P_{\mathscr{C}}^{\mathrm{sep}}$ already is the Néron lft-model of $\Pic^0_{C/K}.$ For nodal curves, a similar question was studied by Orecchia \cite{Ore}. We also investigate the existence of closely related \it semi-factorial models \rm studied by Pépin \cite{Pép}.

\begin{definition}
Let $S$ be a Dedekind scheme with field of rational functions $K.$ A scheme $X\to S$ is \rm semi-factorial \it if the canonical map 
$$\Pic X\to \Pic(X\times_S\Spec K)$$ is surjective.
\end{definition}
This definition can be found in \cite{Pép}, Définition 1.1, where $S$ is assumed to be the spectrum of a discrete valuation ring. We shall also study the following closely related concept: let $C$ be a proper curve over $K$ and let $\mathscr{C}\to S$ be a proper and flat model of $C.$ Let $P_{\mathscr{C}/S}$ be the scheme-theoretic closure\footnote{Note that we can consider these scheme-theoretic closures even if $\Pic_{\mathscr{C}/S}$ is not representable; see \cite{BLR}, p. 265f or \cite{Ray}, (3.2) c).} of $\Pic^0_{C/K},$ let $\mathscr{E}$ be the scheme-theoretic closure of the unit section in $P_{\mathscr{C}/S},$ and put $P_{\mathscr{C}/S}^{\mathrm{sep}}:=P_{\mathscr{C}/S}/\mathscr{E}.$
\begin{definition}
Let $S$ be a Dedekind scheme with field of rational functions $K.$ Let $C$ be a geometrically reduced proper curve over $K.$ A proper and flat model $\mathscr{C}\to S$ is a \rm Néron-Picard model \it of $C$ if the functor $P_{\mathscr{C}/S}^{\mathrm{sep}}$ constructed above is representable and equal to the Néron lft-model of $\Pic^0_{C/K}$ over $S.$ 
\end{definition}
Models of this kind (although not under this name) already appear in \cite{Ore}. If $S$ is excellent and $C$ is geometrically reduced, regular, and each irreducible component of $C$ has a $K$-rational point, then any proper, flat, and regular model $\mathscr{C}\to S$ of $C$ is a Néron-Picard model over $S;$ this follows from Theorem \ref{regmodelexistencetheorem}, Lemma \ref{localgloballem}, and \cite{BLR}, Chapter 9.5, Theorem 4. The main results of this section will be the following: \\
\\
(i) If $S$ is excellent and \it local, \rm then any seminormal proper curve $C$ over $K$ admits a semi-factorial model, and a Néron-Picard model exists if each irreducible component of $C$ has a smooth $K^{\mathrm{sh}}$-rational point, whereas\\
\\
(ii) if $S$ is \it global \rm (i. e., if $S$ has infinitely many closed points) and of finite type over a field, then a geometrically integral proper curve $C$ over $K$ which admits a Néron-Picard model over $S$ is regular. The converse holds if $C$ has a $K$-rational point by Proposition \ref{regularnéronprop}.

\subsection{Semi-factorial models in the local case}
In this paragraph, we shall prove the following result, which partly generalises \cite{Pép}, Théorème 8.1:
\begin{theorem}
Let $S$ be the spectrum of an excellent discrete valuation ring with field of fractions $K.$ Let $C$ be a proper seminormal curve over $K.$ Then $C$ admits a proper, flat, semi-factorial model $\mathscr{C}\to S.$ \label{semfactexthm}
\end{theorem}
We shall need the following stronger version of the embedded resolution theorem:
\begin{proposition}
Let $K$ and $S$ be as in Theorem \ref{semfactexthm}, and let $s$ be the special point of $S.$ Let $\widetilde{C}$ be a proper regular curve over $K.$ Let $\Psi$ be a reduced effective divisor on $\widetilde{C}.$ Then there exists a proper flat model $\widetilde{\mathscr{C}}\to S$ of $\widetilde{C}$ with $\widetilde{\mathscr{C}}$ regular, and a reduced effective divisor ${\overline{\Psi}}$ on $\widetilde{\mathscr{C}}$ extending $\Psi$ such that\\ 
(i) the divisor ${\overline{\Psi}} + \widetilde{\mathscr{C}}_s$ is supported on a divisor with strict normal crossings, where $\widetilde{\mathscr{C}}_s:=\widetilde{\mathscr{C}}\times _S \Spec \kappa(s),$ and \\
(ii) each irreducible component of $\widetilde{\mathscr{C}}_s$ contains at most one point of intersection of ${\overline{\Psi}}$ with $\widetilde{\mathscr{C}}_s,$ and this remains true after strict Henselization of the base. \label{strongembresprop}
\end{proposition}
\begin{proof}
Let $\widetilde{\mathscr{C}}$ be a regular proper flat model of $\widetilde{C},$ which exists by Theorem \ref{regmodelexistencetheorem}. Let $D$ be the scheme-theoretic closure of $\Psi$ in $\widetilde{\mathscr{C}}.$ By the embedded resolution theorem (\cite{Liu}, Chapter 9.2, Theorem 2.26), there exists a proper birational morphism $f\colon \widetilde{\mathscr{C}}'\to \widetilde{\mathscr{C}}$ with $\widetilde{\mathscr{C}}'$ regular, such that $f^\ast (D+ \widetilde{\mathscr{C}}_s)$ is supported on a divisor with strict normal crossings. We put ${\overline{\Psi}}:=f^\ast D$ (note that ${\overline{\Psi}}$ is automatically equal to the integral closure of $S$ in $\Psi$). Now replace $\widetilde{\mathscr{C}}$ by $\widetilde{\mathscr{C}}'.$ Note that the scheme-theoretic intersection ${\overline{\Psi}}\cap (\widetilde{\mathscr{C}}_s)_{\mathrm{red}}$ is a reduced zero-dimensional scheme, and hence regular. Finally, we replace $\widetilde{\mathscr{C}}$ by $\mathrm{Bl}_{{\overline{\Psi}}\cap (\widetilde{\mathscr{C}}_s)_{\mathrm{red}}} \widetilde{\mathscr{C}}.$ Clearly, this is still a regular, proper, and flat model of $\widetilde{C}.$ Moreover, the strict transform of ${\overline{\Psi}}$ is equal to ${\overline{\Psi}}$ because this scheme is regular. Now write
$${\overline{\Psi}}\cap (\widetilde{\mathscr{C}}_s)_{\mathrm{red}}=\{x_1,..., x_n\}$$ for closed points $x_1,..., x_n$ of $\widetilde{\mathscr{C}}.$ Then, by construction, each $x_j$ lies on its own irreducible component $E_j\cong \Ps^1_{\kappa(x_j)}$ of $(\widetilde{\mathscr{C}}_s)_{\mathrm{red}}.$ Hence claim (ii) follows as well.
\end{proof}\\
We can now construct semi-factorial models over excellent discrete valuation rings for arbitrary proper integral seminormal curves: \\
\\
\it Proof of Theorem \ref{semfactexthm}. \rm Throughout this proof, a subscript $\eta$ applied to an object defined over $S$ will denote restriction to the generic fibre. Let $C$ be as in Theorem \ref{semfactexthm} and let $\widetilde{C}$ be the normalisation of $C.$ Let $\Psi$ be the scheme-theoretic pre-image of $C^{\mathrm{sing}}$ in $\widetilde{C}$ (as before, $C^{\mathrm{sing}}$ is the singular locus of $C$ endowed with its reduced subscheme structure). By Theorem \ref{FactorisationTheorem} (ii), $\Psi$ is a reduced effective divisor on $\widetilde{C}.$ Let $\widetilde{\mathscr{C}}$ be the model of $\widetilde{C}$ from Proposition \ref{strongembresprop}. Let $s$ be the special point of $S$ and let $\widetilde{\mathscr{C}}_s$ be the special fibre of $\widetilde{\mathscr{C}}.$ Let $x_1,..., x_n$ be closed points on $\widetilde{\mathscr{C}}$ such that 
$${\overline{\Psi}} \cap (\widetilde{\mathscr{C}}_s)_{\mathrm{red}}=\{x_1,..., x_n\}.$$ In particular, the $x_j$ correspond to maximal ideals $\mathfrak{m}_1,..., \mathfrak{m}_n$ of $\overline{\Psi}.$ For each $j=1,..., n,$ let $E_j$ be the unique irreducible component of $\widetilde{\mathscr{C}}_s$ (endowed with the reduced subscheme structure) on which $x_j$ lies. Let $\mathscr{C}^{\mathrm{sing}}$ be the integral closure of $S$ in $C^{\mathrm{sing}},$ and let $\xi\colon {\overline{\Psi}}\to \mathscr{C}^{\mathrm{sing}}$ be the canonical map. Let $\widetilde{\iota}\colon {\overline{\Psi}}\to \widetilde{\mathscr{C}}$ be the closed immersion. We claim that 
$$\mathscr{C}:=\widetilde{\mathscr{C}}\cup_{{\overline{\Psi}}}\mathscr{C}^{\mathrm{sing}}$$
is a semi-factorial model of $C.$ To see this, recall that a line bundle $\mathscr{L}_\eta$ on $C$ is the same as a triple $(\widetilde{\mathscr{L}}_{\mathscr{\eta}}, \mathscr{N}_{\eta},\lambda_\eta),$ where $\widetilde{\mathscr{L}}_{\eta}$ and $\mathscr{N}_{\eta}$ are line bundles on $\widetilde{C}$ and ${C}^{\mathrm{sing}},$ respectively, and 
$$\lambda_\eta \colon \widetilde{\iota}_\eta^\ast \widetilde{\mathscr{L}}_\eta \to \xi_\eta^\ast \mathscr{N}_\eta$$ is an isomorphism (see Proposition \ref{linebundlesequivprop}). A similar description holds for line bundles on $\mathscr{C}.$ \\
Let $\mathscr{L}_\eta$ be a line bundle on $C.$ Let $\psi_\eta\colon \widetilde{{C}}\to {C}$ be the normalisation morphism, and let $\iota\colon C^{\mathrm{sing}}\to C$ be the canonical closed immersion. Because $\widetilde{\mathscr{C}}$ is regular, we can extend $\psi_\eta ^ \ast \mathscr{L}_\eta$ to a line bundle $\widetilde{\mathscr{L}}$ on $\widetilde{\mathscr{C}}.$ Let $\sigma_\eta$ be a no-where vanishing global section of $\iota_\eta^\ast \mathscr{L}_\eta.$ We may choose integers $\nu_1, ..., \nu_n$ such that $\psi_\eta^\ast \sigma_\eta$ extends to a no-where vanishing global section of
$$\widetilde{\iota}^\ast \widetilde{\mathscr{L}}\otimes_{\overline{\Psi}} (\mathfrak{m}_1^{-\nu_1} \cdot ... \cdot \mathfrak{m}_n^{-\nu_n}) = \widetilde{\iota}^\ast \widetilde{\mathscr{L}}(\nu_1 E_1 + ... + \nu_nE_n).$$
This no-where vanishing global section defines an isomorphism 
$$\lambda\colon \widetilde{\iota}^\ast \widetilde{\mathscr{L}}(\nu_1 E_1 + ... + \nu_nE_n) \to \Og_{\overline{\Psi}} = \xi^\ast \Og_{\mathscr{C}^{\mathrm{sing}}}.$$ In particular, if $\mathscr{L}$ is the line bundle on $\mathscr{C}$ corresponding to the triple $$(\widetilde{\mathscr{L}}(\nu_1 E_1 + ... + \nu_nE_n), \Og_{\mathscr{C}^{\mathrm{sing}}}, \lambda)$$ according to Proposition \ref{linebundlesequivprop}, then $\mathscr{L}$ extends $\mathscr{L}_\eta,$ as claimed. 
\qed
\subsection{Néron-Picard models in the local case}
We keep the notation from the previous Paragraph. If $K$ is the field of fractions of a discrete valuation ring $R,$ we denote by $K^{\mathrm{sh}}$ the field of fractions of the strict Henselisation $R^{\mathrm{sh}}$ of $R$ with respect to a separable closure of the residue field of $R$ (see \cite{Stacks}, Tag 0BSK). We shall now prove
\begin{theorem}
Let $S$ be the spectrum of an excellent discrete valuation ring with field of fractions $K.$ Let $C$ be a proper geometrically reduced seminormal curve such that each irreducible component of $C$ contains an element of $C^{\mathrm{sm}}(K^{\mathrm{sh}}).$ Then $C$ admits a Néron-Picard model $\mathscr{C}\to S.$ \label{NérPicexthm} Moreover, the model we construct remains a Néron-Picard model after any essentially smooth or ind-étale extension of $\Gamma(S, \Og_S)$ of discrete valuation rings. 
\end{theorem}
We need the following technical
\begin{lemma}
Let $(X_i)_{i\in I}$ be a directed system of Noetherian schemes with étale affine transition maps. Let $0\in I$ be an element and let $D\subseteq X_0$ be a divisor with strict normal crossings. Let $X:=\varprojlim X_i$, assume that $X$ is Noetherian, and let $\pi_0\colon \varprojlim X_i \to X_0$ be the projection morphism. Then $\pi_0^\ast D$ is a divisor with strict normal crossings on $X.$ \label{limitlem}
\end{lemma}
\begin{proof}
We may assume without loss of generality that $i\geq 0$ for all $i\in I.$ For $i\geq j$ in $I,$ let $\tau_{ij}\colon X_i\to X_j$ the transition map. Let $x\in \pi_0^\ast D.$ For each $i\in I,$ let $\pi_i\colon X\to X_{i}$ be the projection. First observe that we have 
$$\Og_{X,x}=\varinjlim \Og_{X_i, \pi_i(x)}.$$ Now choose a regular system of parameters $x_1,..., x_d$ such that $D$ is cut out by $x_1\cdot...\cdot x_r$ in $\Spec \Og_{X_0, \pi_0(x)}$ for some $1\leq r\leq d.$ For each $i\in I,$ let $\mathfrak{m}_i$ denote the maximal ideal of $\Og_{X_i, \pi_i(x)}.$ Note that the $x_1,..., x_d$ generate $\mathfrak{m}_i$ for all $i$  because the transition maps are étale. By considering the chain $0\subset \langle x_1 \rangle \subset \langle x_1, x_2,\rangle \subset ... \subset \langle x_1,..., x_d \rangle$ of prime ideals in $\Og_{X,x},$ we see that $\Og_{X,x}$ is regular. In particular, $x_1,..., x_d$ is a regular regular system of parameters in $\Og_{X,x}$, and $\pi_0^\ast D$ is cut out by $x_1\cdot...\cdot x_r.$ Hence the claim follows.  
\end{proof}\\
\it Proof of Theorem \ref{NérPicexthm}. \rm We construct a model $\mathscr{C}\to S$ as in the proof of Theorem \ref{semfactexthm}, and we shall use the notation introduced there without introducing it again. Observe that we may assume without loss of generality that $S$ is strictly Henselian. First, we show that $\mathscr{C}\to S$ is cohomologically flat in dimension zero. Since $\mathscr{C}=\widetilde{\mathscr{C}}\cup_{{\overline{\Psi}}}\mathscr{C}^{\mathrm{sing}},$ we already know that $\mathscr{C}$ is proper and flat over $S.$ Moreover, $\widetilde{\mathscr{C}}\to S$ is cohomologically flat in dimension zero by \cite{Liu}, Chapter 9.1, Corollary 1.24. 
Indeed, we may suppose that $\widetilde{\mathscr{C}}$ is irreducible. Then we know that of $\widetilde{\mathscr{C}}$ admits a section. In particular, $\widetilde{\mathscr{C}}\times_S\Spec \kappa(s)$ has a smooth $\kappa(s)$-point. This implies that $\widetilde{\mathscr{C}}\times_S\Spec \kappa(s)$ has an irreducible component of geometric multiplicity 1. Therefore $\mathscr{C}$ is cohomologically flat in dimension zero over $S$ by Lemma \ref{cohomflatlem}. This means that $P_{\mathscr{C}/S}^{\mathrm{sep}}$ is scheme which is smooth and locally of finite presentation over $S.$\\
To show that $P_{\mathscr{C}/S}^{\mathrm{sep}}$ is the Néron model of $\Pic^0_{C/K},$ it suffices to show the following (Proposition \ref{strongcritprop}): for each discrete valuation ring $R$ which is essentially smooth over $\Gamma(S, \Og_S)$ and any strict Henselisation $R^{\mathrm{sh}}$ of $R$, the morphism 
$$\Pic_{\mathscr{C}/S}(R^{\mathrm{sh}}) \to \Pic_{\mathscr{C}/S}(F^{\mathrm{sh}})$$ is surjective, where $F^{\mathrm{sh}}:=\mathrm{Frac}\, R^{\mathrm{sh}}.$ Choose such an $R$ and let $S':=\Spec R^{\mathrm{sh}}.$ Then $\widetilde{\mathscr{C}}\times_SS'\to S'$ still satisfies the conditions (i) and (ii) from Proposition \ref{strongembresprop}. This follows from \cite{Stacks}, Tag 0CBP together with Lemma \ref{limitlem}. Since $S'$ is strictly Henselian, the map 
$$\Pic(\mathscr{C}\times_SS')\to \Pic_{\mathscr{C}/S}(R^{\mathrm{sh}})$$ is an isomorphism. The map $\Pic(\mathscr{C}\times_S\Spec F^{\mathrm{sh}})\to \Pic_{\mathscr{C}/S}(F^{\mathrm{sh}})$ is surjective because of our condition on $K^{\mathrm{sh}}$-points. The proof of Theorem \ref{semfactexthm} shows that the canonical map 
$$\Pic(\mathscr{C}\times_SS')\to\Pic(\mathscr{C}\times_S\Spec F^{\mathrm{sh}})$$ is surjective, so our claim follows. \qed\\
\\
$\mathbf{Example}.$ Let $K$ be the field of fractions of a discrete valuation ring $R.$ Let us construct a model of the curve $C$ given by identifying two distinct $K$-points of $\Ps^1_K.$ This curve is isomorphic to the push-out of $\Ps^1_K$ along the map $\Spec K\sqcup \Spec K\to \Spec K$ along the closed immersion $\Spec K\sqcup \Spec K\to \Ps^1$ whose image are $[1:0]$ and $[0:1].$ We choose the canonical model $\Ps^1_R$ of $\Ps^1_K.$ We can extend the map $\Spec K\sqcup \Spec K\to \Ps^1$ to a closed immersion 
$$\Spec R\sqcup \Spec R \to \Ps^1_R.$$ Let $\mathscr{C}$ be the model of $C$ obtained by the push-out
$$\begin{CD}
\Ps^1_R@>>>\mathscr{C}\\
@AAA@AAA\\
\Spec R\sqcup \Spec R@>>>\Spec R.
\end{CD}$$
Now condition (i) from Proposition \ref{strongembresprop} is already satisfied (otherwise we would have to blow up points on the special fibre to obtain a divisor with strict normal crossings). However, condition (ii) is not satisfied. Proposition \ref{exactsequenceProp} gives us an exact sequence
$$0\to {\mathbf{G}_{\mathrm{m}, R}} \to \Pic_{\mathscr{C}/R}\to \Pic_{\Ps^1_{R}/R}=\Z\to 0,$$ which induces an isomorphism
$$P_{\mathscr{C}/\Spec R}\cong {\mathbf{G}_{\mathrm{m}, R}}.$$ In particular, $\mathscr{C}$ is not Néron-Picard. 
Now we proceed as described above: let $\widetilde{\mathscr{C}}$ be the model of $\Ps^1_K$ obtained by blowing up the north and south pole of the special fibre of $\Ps^1_R.$ Let $\mathscr{C}'$ be the model obtained by the push-out
$$\begin{CD}
\widetilde{\mathscr{C}}@>>>\mathscr{C}'\\
@AAA@AAA\\
\Spec R\sqcup \Spec R@>>>\Spec R.
\end{CD}$$
Observe that $P_{\widetilde{\mathscr{C}}/\Spec R}$ is étale over $S$ (it is trivial generically, so this follows from \cite{Ray}, Proposition 3.3.5). Using the snake lemma, we derive an exact sequence
$$0\to {\mathbf{G}_{\mathrm{m}, R}}\to P^{\mathrm{sep}}_{\mathscr{C}'/\Spec R}\to \mathscr{Q}\to 0$$ over $R,$ where $\mathscr{Q}$ is a quotient of $P_{\widetilde{\mathscr{C}}/\Spec R},$ and hence étale over $R$. (We know from Theorem \ref{NérPicexthm} and its proof that $P^{\mathrm{sep}}_{\mathscr{C}'/\Spec R}$ is the Néron lft-model of its generic fibre, so we see that $\mathscr{Q}\cong i_\ast \Z,$ where $i$ is the inclusion of the special point of $\Spec R.$) This is exactly what we expect form the Néron lft-model of ${\mathbf{G}_{\mathrm{m}, K}}.$ This example illustrates that it is precisely the additional non-separatedness of $P_{\widetilde{\mathscr{C}}/S}$ which makes this construction possible. \\
\\
$\mathbf{Remark}.$ Let $R$ be a discrete valuation ring with field of fractions $K$ and let $\mathscr{C}\to S:=\Spec R$ be a proper and flat morphism whose fibres are \it nodal curves with split singularities \rm (\cite{Ore}, Definitions 1.1 and 1.2). Orecchia \cite{Ore} studied the question when $\mathscr{C}$ is a Néron-Picard model of its generic fibre. Basically, his result (as stated in \cite{Ore}) can be paraphrased as follows: we let $\Gamma$ be the dual graph of the special fibre of $\mathscr{C}\to S.$ We consider the \it labelled graph $(\Gamma, l)$ of \rm $\mathscr{C}\to S,$ where each edge of $\Gamma$ is labelled by the \it thickness \rm of the corresponding singularity of the special fibre (see \cite{Ore}, Definition 6.1). The thickness measures \it how singular \rm a singularity of the special fibre is when considered as a point of $\mathscr{C}.$ This is an element of $\N\cup \{\infty\}$ which is equal to 1 if and only if  the corresponding point of $\mathscr{C}$ is regular, and equal to $\infty$ if and only if the corresponding point on $\mathscr{C}$ is a specialisation of a node on the generic fibre of $\mathscr{C}.$ We call $(\Gamma, l)$ \it circuit-coprime \rm if the labels appearing in any circuit in the resulting graph have no common prime divisor (see the Definition in \cite{OreII} which replaces \cite{Ore}, Definition 5.19). Now Theorem 7.6 of \cite{Ore} states that $\mathscr{C}\to S$ is a Néron-Picard model of its generic fibre if and only if $(\Gamma, l)$ is circuit-coprime (this is true as originally stated with the new definition of circuit-coprimality from \cite{OreII}). \\
The example from above illustrates this result. Indeed, let $\mathscr{C}$ be the model of $C$ from the Example above (constructed in the first push-out diagram). Then we have an  
isomorphism $P_{\mathscr{C}/R}\cong {\mathbf{G}_{\mathrm{m}, R}},$ which is not the Néron lft-model of its generic fibre. This is explained by the fact that the labelled graph associated with this model is not circuit-coprime. The labelled graph associated with the second model $\mathscr{C}'\to \Spec R$ we constructed above is circuit-coprime, as can be easily calculated.  

\subsection{The global case}
Now let $S$ be a regular connected algebraic curve over a field $\kappa.$ Let $K$ be the field of rational functions of $S.$
\begin{proposition}
Let $C$ be a proper, geometrically integral curve over $K$ and suppose that $C$ admits a Néron-Picard model $\mathscr{C}\to S.$ Then $C$ is regular. 
\end{proposition}
\begin{proof}
Because $C$ is projective over $K$ (\cite{Stacks}, Tag 0A26), we can find a dense open subset of $S$ above which $\mathscr{C}$ is projective. We replace $S$ by that dense open subset and assume that $\mathscr{C}\to S$ is a projective morphism. Shrinking $S$ further, we may assume that the morphism $\mathscr{C}\to S$ has geometrically integral fibres. Then \cite{BLR}, Chapter 9.3, Theorem 1 tells us that $P_{\mathscr{C}}$ is separated (i. e., $\mathscr{E}=0$) and has connected fibres. By \cite{BLR}, Chapter 10.1, Corollary 10, $\Pic^0_{C/K}$ admits a Néron model (of finite type), so we must have $\mathrm{uni}(\Pic^0_{C/K})=0$ (\cite{BLR}, Chapter 10.3, Theorem 5; this is where we use that $S$ is a curve over a field). By Corollary \ref{Jacstructcor}, the Jacobian of $C$ is isomorphic to that of its normalisation. Because $C$ is geometrically integral, this implies that $C$ is regular (otherwise the morphism $\Pic^0_{C/K} \to \Pic^0_{\widetilde{C}/K}$ would have non-trivial kernel). 
\end{proof}\\
$\mathbf{Remark}.$ In the light of this Proposition, it seems reasonable to expect that Néron-Picard models exist over a global Dedekind scheme only for regular curves, at least in the geometrically integral case. On the other hand, the existence of semi-factorial models in the global case appears to be a much more delicate problem, and it would be very interesting to gain some insight in this regard. For example, it does not even seem to be clear whether any singular curves admit semi-factorial models in the global case. 
\subsection{Some open questions}
Finally, let us mention a few more questions which this article leaves open. First of all, we have proven Conjecture I and Conjecture II from \cite{BLR}, Chapter 10.3 only for Jacobians of \it geometrically reduced \rm curves. However, both Conjectures make claims for general smooth group schemes of finite type over fields. While the present proof is confined to the world of Jacobians for obvious reasons, the assumption that $C$ is geometrically reduced could very well be an artefact of our proof, and it would be fascinating to know whether this condition can be removed within the boundaries of the present methods. It should be noted, however, that such a generalisation would most probably not be straightforward, as the connection between Picard functors and Néron models so-far seems to require that the curves be geometrically reduced, even in the regular case (see \cite{BLR}, Chapter 9.5, Theorem 4).  \\
For example, we proved that if $\widetilde{C}$ is a geometrically reduced \it regular \rm curve over a field $\kappa,$ then $\mathrm{uni}(\Pic^0_{\widetilde{C}/\kappa})=0$ (see Proposition \ref{unitrivialprop}). The proof of this Proposition which we gave uses that $\widetilde{C}$ is geometrically reduced in an essential way. All attempts to resolve this problem so far ended up involving very difficult problems about Brauer groups over imperfect fields. It would already be interesting to know the answer to\\
\\
$\mathbf{Question\ 1}.$ \it Let $\widetilde{C}$ be a (not necessarily geometrically reduced) regular proper curve over a field $\kappa.$ Is it true that $\mathrm{uni}(\Pic^0_{\widetilde{C}/\kappa})=0?$ If not, is it true that $\mathscr{R}_{us, \kappa}(\Pic^0_{\widetilde{C}/\kappa})=0?$ \rm\\
\\
One can reduce this question to the case where $\Gamma(\widetilde{C}, \Og_{\widetilde{C}})=\kappa$ in a relatively straightforward manner, but beyond that, almost nothing seems to be known. It should also be noted that Conjecture I and Conjecture II quoted above do not seem to be known in general for Jacobians of regular curves.
Moreover, we have seen that semi-factorial models and Néron-Picard models exist in the local case for a rather large class of curves, whereas the situation is much less clear for global bases. This leads to\\
\\
$\mathbf{Question\ 2}.$ \it Let $S$ be a Dedekind scheme and let $C$ be a geometrically integral proper curve over $K.$\\
(i) Suppose that $S$ is global and that $C$ admits a Néron-Picard model over $S.$ Does it follow that $C$ is regular? \\
(ii) Suppose that $C$ admits a semi-factorial model over $S.$ If $S$ is local, does that imply that $C$ is seminormal? If $S$ is global, does it follow that $C$ is regular?
\rm

\end{document}